%% file: main.tex
\documentclass{article}

\usepackage{subfigure}
\usepackage{amsmath, amssymb, mathtools, amsfonts}
\usepackage{graphicx}
\usepackage{epstopdf}
\usepackage{algorithm}
\usepackage{algorithmic}
\usepackage{amsopn}
\usepackage{appendix}
\usepackage{hyperref}
\usepackage{xcolor}
\usepackage[normalem]{ulem}
\usepackage[
    natbib=true,
    style=numeric,
    sorting=nty
]{biblatex}
\usepackage[normalem]{ulem}

\title{A parallel-in-time collocation method using diagonalization: theory and implementation for linear problems.\thanks{Submitted to the editors on \today.}}

\author{Gayatri \v{C}aklovi\'{c}
\footnotemark[2] \footnotemark[3]
\and Robert Speck\thanks{g.caklovic@fz-juelich.de, r.speck@fz-juelich.de, Jülich Supercomputing Centre, Forschungszentrum Jülich, 52425 Jülich, Germany.}
\and Martin Frank\thanks{gayatri.caklovic@kit.edu, martin.frank@kit.edu, Department of Mathematics, Karlsruhe Institute of Technology, 76049 Karlsruhe, Germany}}

\begin{document}

\maketitle
\begin{abstract}
We present and analyze a parallel implementation of a parallel-in-time collocation method based on $\alpha$-circulant preconditioned Richardson iterations.
While many papers explore this family of single-level, time-parallel ``all-at-once'' integrators from various perspectives, performance results of actual parallel runs are still scarce.
This leaves a critical gap, because the efficiency and applicability of any parallel method heavily rely on the actual parallel performance, with only limited guidance from theoretical considerations.
Further, challenges like selecting good parameters, finding suitable communication strategies, and performing a fair comparison to sequential time-stepping methods can be easily missed.
In this paper, we first extend the original idea of these fixed point iterative approaches based on $\alpha$-circulant preconditioners to high-order collocation methods, adding yet another level of parallelization in time "across the method".
We derive an adaptive strategy to select a new $\alpha$-circulant preconditioner
for each iteration during runtime for balancing convergence rates, round-off errors, and inexactness of inner system solves for the individual time-steps.
After addressing these more theoretical challenges, we present an open-source space- and time-parallel implementation and evaluate its performance for two different test problems.
\end{abstract}



\input{section1}
\input{section2}
\input{section2a}
\input{section3}
\input{section4}
\input{section5}
\input{section6}

\printbibliography

\appendix
\input{appendix}
\end{document}

%% file: section1.tex
\input{macros}

\section{Introduction}
Many flavors of numerical methods for computing the solution to an initial value problem exist.
Usually, this is achieved in a sequential way, in the sense that an approximation of the solution at the time-step comes from propagating an approximation at an earlier point in time with some numerical method that depends on the time-step size. 
Nowadays, parallelization of these algorithms in the temporal domain is gaining more and more attention due to the growing number of available resources, parallel hardware, and, recently, more and more mathematical approaches. 
First theories and proposals of parallelizing this process emerged in the 60s~\cite{firstpint} and got significant attention after the \textit{Parareal} paper was published in 2001~\cite{parareal}. 
As one of the key ideas, \textit{Parareal} introduced coarsening in time to reduce the impact of the sequential nature of time-stepping.

Parallel-in-Time (PinT) methods have been shown to be useful for quite a range of applications and we refer to the recent overview in~\cite{OngEtAl2020} for more details. 
The research area is very active~\cite{pint_ref}, with many covered in the seminal work '50 Years of Time Parallel Time Integration'~\cite{50y}. 
However, while the theoretical analysis is rather advanced, only a few parallel implementations on modern high-performance computing systems exist. 
Notable exceptions are the implementations of multigrid reduction in time~\cite{MGRIT,xbraid-package,HahneEtAl2020}, the parallel full approximation scheme in space and time~\cite{EmmettMinion2012,pfasst_code,Speck2019}, and the revisionist integral deferred correction method~\cite{RIDC}.

More often than not, expected and unexpected pitfalls occur when numerical methods are finally moved from a theoretical concept to actual implementation, let alone on parallel machines.
Sole convergence analysis of a method is often not enough to prove usefulness in terms of parallel scaling or reduced time-to-solution~\cite{12ways}. 
For example, the mathematical theory has shown that the Parareal algorithm (as many other parallel-in-time methods) is either unstable or inefficient when applied to hyperbolic problems~\cite{pararealeigenvalues, pararealadvection, pararealfluids}, therefore new ways to enhance it were developed~\cite{pararealenhanced1, pararealenhanced2, pararealenhanced3}. 
Yet, most of these new approaches have more overhead by design, leading to less efficiency and applicability. 
As an example, there is a phase-shift error arising from mismatches between the phase speed of coarse and fine propagator when solving hyperbolic problems~\cite{rup}.

As an alternative, there exist methods such as paraexp~\cite{paraexp} and REXI~\cite{rexi} which do not require coarsening in time.
Another parallel-in-time approach, which can be used without the need to coarsen, are iterative methods based on $\alpha$-circulant preconditioners for ``all-at-once'' systems of partial differential equations (PDEs). 
This particular preconditioner is diagonalizable, resulting in decoupled systems across the time-steps, and early results indicate that these methods will also work for hyperbolic problems.
The $\alpha$-circulant preconditioner can be used to accelerate GMRES or MINRES solvers~\cite{mcdonlad2016, mcdonald2018, epsilonpeople} while other variants can be interpreted as special versions of Parareal \cite{paradiag} or preconditioned Richardson iterations~\cite{slw2021}.
In general, parallel-in-time solvers relying on the diagonalizable $\alpha$-circulant preconditioner can be categorized as ``ParaDiag'' algorithms \cite{pint_web, para_DIAG, para_DIAG_simax}.
While many of these approaches show very promising results from a theoretical point of view, their parallel implementations are often not demonstrated and can be challenging to implement efficiently in parallel.
Selecting parameters, finding efficient communication strategies, and performing fair comparisons to sequential methods~\cite{12ways} are only three of many challenges in this regard.  

The aim of this paper is to close this important gap for the class of $\alpha$-circulant preconditioners, which is gaining more and more attention in the field.
Using the preconditioned Richardson iterations with an $\alpha$-circulant preconditioner as presented in~\cite{slw2021}, we describe the path from theory to implementation for these methods in the case of linear problems.
We extend the original approach by using a collocation method as our base integrator and show that this choice not only yields time integrators of arbitrary order but also adds another level (and opportunity) of parallelization in time.
Following the taxonomy in~\cite{Burrage1997}, this yields a doubly-time parallel method: combining parallelization across the steps and parallelization across the method first proposed in~\cite{SchoebelEtAl2020}.
We present and analyze an adaptive strategy to select a new parameter $\alpha$ for each iteration in order to balance (1) the convergence rate of the method with (2) round-off errors arising from the diagonalization of the preconditioner and (3) inner system solves of the decoupled time-steps. 
After some error analysis, we present an algorithm for the practical selection of the adaptive preconditioners and demonstrate the actual speedup gains these methods can provide over sequential methods. 
These experiments provide some insight into more complex applications and variants using the diagonalization technique, narrowing the gap between real-life applications and theory.
The code is freely available on GitHub~\cite{paralphacode}, along with a short tutorial and all necessary scripts to reproduce the results presented in this manuscript.

%% file: macros.tex
\newcommand{\R}{\mathbb{R}}
\newcommand{\C}{\mathbb{C}}
\newcommand{\N}{\mathbb{N}}
\newcommand{\vect}[1]{\mathbf{#1}}
\newcommand{\tvect}[1]{\vec{\vect{#1}}}
\newcommand{\alert}[1]{{\color{red} #1}}

\newcommand{\coll}{{\operatorname{coll}}}
\newcommand{\tA}{\hat{\vect{A}}}
\newcommand{\tB}{\hat{\vect{B}}}
\newcommand{\tC}{\hat{\vect{C}}}
\renewcommand{\u}{\underline{u}}
\newcommand{\sm}{\text{-}}
\newcommand{\nproc}{n_{\operatorname{proc}}}
\newcommand{\nstep}{n_{\operatorname{step}}}
\newcommand{\ncoll}{n_{\operatorname{coll}}}
\newcommand{\nspace}{n_{\operatorname{space}}}
\newcommand{\tsys}{T_{\operatorname{sys}}}
\newcommand{\tcomm}{T_{\operatorname{comm}}}
\newcommand{\tseq}{T_{\operatorname{seq}}}
\newcommand{\tpar}{T_{\operatorname{par}}}
\newcommand{\diag}{\operatorname{diag}}

\newtheorem{theorem}{Theorem}
\newtheorem{corollary}{Corollary}
\newtheorem{remark}{Remark}
\newtheorem{lemma}{Lemma}
\newtheorem{example}{Example}
\newtheorem{proposition}{Proposition}
\newtheorem{definition}{Definition}
\newenvironment{proof}{{\sc Proof:}}{~\hfill $\square$\\}

\renewcommand\algorithmiccomment[1]{\hfill\#\ {#1}}

%% file: section2.tex
\section{The method}
In this section, we briefly explain the derivation of our base numerical propagator: the collocation method. 
After that, the formulation of the so-called ``composite collocation problem'' for multiple time-steps is described.
Preconditioned Richardson iterations are then used to solve the composite collocation problem parallel-in-time.

In this paper, we focus on linear initial value problems on some interval $[0, T]$, generally written as
\begin{align}\label{eq:genericlinearform}
	u_t &= \vect{A}u + b(t), \quad u(0) = u_0,
\end{align}
where $\vect{A} \in \C^{N \times N}$ is a constant matrix, $b:\R \rightarrow \C^N$ and $u_0 \in \C^N$, $N\in\mathbb{N}$. 
Let $0 = T_0 < T_1 < \dots < T_L = T$ be an equidistant subdivision of the time interval $[0, T]$ with a constant step size $\Delta T$ .
\subsection{The composite collocation problem}
Let $0 < t_1 < t_2 < \dots < t_M \leq \Delta T$ be a subdivision of $[0, \Delta T]$.
Without loss of generality, we can assume $\Delta T = 1$. 
Under the assumption that $u \in \C^1([0, 1])$, an integral formulation of~\eqref{eq:genericlinearform} in these nodes is
\begin{align}\label{eq:integral}
	u(t_m) &= u(0) + \int_0^{t_m}\big(\vect{A}u(s) + b(s)\big)ds, \quad m = 1, \dots, M.
\end{align}
Approximating the integrand as a polynomial in these nodes, this yields
\begin{align}\label{eq:integralapproximation}
	\int_0^{t_m}\big(\vect{A}u(s) + b(s)\big)ds \approx \sum_{i=1}^M \bigg(\int_0^{t_m}c_i(s)ds \bigg)\big(\vect{A}u(t_i) + b(t_i)\big),
\end{align}
where $c_i$ denotes the $i$th Lagrange polynomial, for $i=1, \dots, M$. 
Because of the quadrature approximation~\eqref{eq:integralapproximation}, the integral equation~\eqref{eq:integral} can be rewritten in matrix formulation, generally known as the \textit{collocation problem}:
\begin{align}\label{eq:collocation}
	(\vect{I}_{MN} - \vect{Q} \otimes \vect{A})\vect{u} &= \vect{u}_0 + (\vect{Q} \otimes \vect{I}_N) \vect{b},
\end{align}
where $\vect{Q} \in \R^{M \times M}$ is a matrix with entries $q_{mi} = \int_0^{t_m}c_i(s)ds$, $\vect{u} = (u_1, \dots, u_M)\in\mathbb{C}^{MN}$ is the approximation of the solution $u$ at the collocation nodes $t_m$, $\vect{u}_0 = (u_0, \dots, u_0)\in\mathbb{C}^{MN}$, and $\vect{b} = (b_1, \dots, b_M)\in\mathbb{C}^{MN}$ is a vector of the function $b$ evaluated at the nodes.
If one wants to solve the initial value problem on an interval of length $\Delta T$, the matrix $\vect{Q}$ in~\eqref{eq:collocation} is replaced with $\Delta T \vect{Q}$, which is verified by the corresponding change of variables in the integrals stored in $q_{mi}$. 

The collocation problem has been well-studied~\cite{numODE}. 
Throughout this work, we will use Gauss-Radau quadrature nodes, with the right endpoint included as a collocation node.
This gives us a high-order method with errors behaving as $\|u(t_M) - u_M\| = O(\delta t^{2M-1})$, where $\delta t = \operatorname{max}_m(t_{m+1} - t_m)$~\cite{volterra}. 

Now, let us define $\vect{C}_\coll := \vect{I}_{MN} - \Delta T \vect{Q} \otimes \vect{A} \in \C^{MN \times MN}$ and let $\vect{H} := \vect{H}_M \otimes \vect{I}_N$, where 
\begin{equation}\label{eq:Hm}
    \vect{H}_M = 
    \begin{bmatrix}
        0 & \dots & 1 \\
        \vdots & & \vdots \\
        0 & \dots & 1
    \end{bmatrix}
    \in \R^{M\times M}.
\end{equation}
A classical (sequential) approach to solve equation~\eqref{eq:genericlinearform} with the collocation method for L time steps is first solving
\begin{equation*}
    \vect{C}_\coll \vect{u}_1 = \vect{u}_0 + \vect{v}_1,
\end{equation*}
and then sequentially solving
\begin{equation}\label{eq:seq_collocation_problem}
    \vect{C}_\coll \vect{u}_\ell =\vect{H}\vect{u}_{\ell - 1} + \vect{v}_{\ell - 1},
\end{equation}
for $\ell = 2, \dots, L$.
Here, the vectors $\vect{u}_l\in\mathbb{C}^{MN}$ represent the solution on $[T_l, T_{l+1}]$, and $\vect{v}_l = \Delta T (\vect{Q} \otimes \vect{I}_N)\vect{b}_l$, where $\vect{b}_l$ is the function $b$ evaluated at the corresponding collocation nodes $T_l+t_1 < T_l + t_2 < \dots < T_l + t_M = T_{l+1}$. 
The matrix $\vect{H}$ serves to utilize the previously obtained solution $\vect{u}_{\ell-1}$ as the initial condition for computing $\vect{u}_\ell$.
For $L$ time-steps, equation~\eqref{eq:seq_collocation_problem} can be cast as an "all-at-once" or a "composite collocation" system:

\begin{align*}
\begin{bmatrix}
\vect{C}_\coll 	& 		& 		& 		\\
\sm \vect{H} 	& \vect{C}_\coll & 		& 		\\
		&\ddots		&\ddots		&		\\
		&		&\sm \vect{H}	&\vect{C}_\coll 	\\
\end{bmatrix}
\begin{bmatrix}
	\vect{u}_1\\
	\vect{u}_2\\
	\vdots \\
	\vect{u}_L\\
\end{bmatrix}
&=
\begin{bmatrix}
	\vect{u}_0 + \vect{v}_1\\
	\vect{v}_2\\
	\vdots \\
	\vect{v}_L\\
\end{bmatrix}.
\end{align*}
More compactly, this can be written as 
\begin{align}\label{eq:globalsys}
    \vect{C}\tvect{u} := (\vect{I}_L\otimes\vect{C}_\coll + \vect{E}\otimes\vect{H})\tvect{u} = \tvect{w}
\end{align}
with $\tvect{u} = \left(\vect{u}_1, \vect{u}_2, \dots, \vect{u}_L\right)^T\in\mathbb{C}^{LMN}$, $\tvect{w} = \left(\vect{u}_0 + \vect{v}_1, \vect{v}_2, \dots, \vect{v}_L\right)^T\in\mathbb{C}^{LMN}$, where the matrix $\vect{E}\in\mathbb{R}^{L\times L}$ has $\sm 1$ on the lower sub-diagonal and zeros elsewhere, accounting for the transfer of the solution from one step to another.

\subsection{The ParaDiag approach for the composite collocation problem}
The system~\eqref{eq:globalsys} is well-posed as long as~\eqref{eq:collocation} is well-posed.
In order to solve it in a simple, yet efficient way, one can use preconditioned Richardson iterations where the preconditioner is a diagonalizable block-circulant Toeplitz matrix defined as
\begin{align}\label{eq:prec}
\vect{C}_\alpha &:= \vect{I}_L \otimes \vect{C}_\coll + \vect{E}_\alpha \otimes \vect{H}, \quad
\vect{E}_\alpha := 
\begin{bmatrix}
0 	& 		& 		&-\alpha	\\
-1 	& 0 & 		& 		\\
		&\ddots		&\ddots		&		\\
		&		& -1	&0	\\
\end{bmatrix},
\end{align}
for $\alpha > 0$.
This yields an iteration of the form
\begin{align}\label{eq:Riteration}
 \vect{C}_\alpha \tvect{u}^{(k+1)} &= (\vect{C}_\alpha - \vect{C})\tvect{u}^{(k)} + \tvect{w}.
\end{align}
Conveniently, $\vect{E}_\alpha$ can be diagonalized.
\begin{lemma}\label{lemma:diag}
The matrix $\vect{E}_\alpha \in \R^{L \times L}$ defined in~\eqref{eq:prec} can be diagonalized as $\vect{E}_\alpha = \vect{V} \vect{D} \vect{V}^{-1}$, where $\vect{V} = \frac{1}{L}\vect{J} \vect{F}$ and
\begin{gather*}
\vect{D}_\alpha = \operatorname{diag}(d_1, \dots, d_L), \, d_l = -\alpha^{\frac{1}{L}}e^{-2\pi i\frac{l-1}{L}}, \\
\vect{J} = \operatorname{diag}(1, \alpha^{-\frac{1}{L}}, \dots, \alpha^{-\frac{L-1}{L}}) \\
[\vect{F}]_{jk} = e^{2\pi i\frac{(j-1)(k-1)}{L}}, \, 1 \leq j, k \leq L, \\
\vect{V}^{-1} = \vect{F}^*\vect{J}^{-1}.
\end{gather*}
\end{lemma}
\begin{proof}
The proof can be found in~\cite{inversetoeplitz}.
\end{proof}
The diagonalization property of $\vect{E}_\alpha$ involving the scaled Fourier matrix allows parallelization across time-steps~\cite{Burrage1997} and has been previously analyzed~\cite{mcdonlad2016, mcdonald2018, epsilonpeople, paradiag, slw2021}. 
Using the diagonalization property, each iteration in~\eqref{eq:Riteration} can be computed in three steps:
\begin{subequations} \label{eq:s123}
\begin{gather}
\tvect{x} = (\vect{V}^{-1} \otimes \vect{I}_{MN})((\vect{C}_\alpha - \vect{C})\tvect{u}^{(k)} + \tvect{w}), \label{eq:s1} \\
(\vect{D}_\alpha \otimes \vect{H} + \vect{I}_L \otimes \vect{C}_\coll)\tvect{y} = \tvect{x}, \label{eq:s2} \\
\tvect{u}^{(k+1)} = (\vect{V} \otimes \vect{I}_{MN}) \tvect{y}\label{eq:s3}.
\end{gather} 
\end{subequations}
Step~\eqref{eq:s2} can be performed in parallel since the matrix is block-diagonal while steps~\eqref{eq:s1} and~\eqref{eq:s3} can be computed with a parallel Fast Fourier transform (FFT) in time.

The block diagonal problems in equation~\eqref{eq:s2} can be expressed as
\begin{align*}
    (d_l \vect{H} + \vect{C}_\coll)\vect{y}_l = \vect{x}_l, \quad l=1, \dots L,
\end{align*}
for $d_l\in\mathbb{C}$, or in other words,
\begin{align}\label{eq:inner1}
\big((d_l\vect{H}_M + \vect{I}_{M}) \otimes\vect{I}_{N} - \Delta T \vect{Q} \otimes \vect{A}\big)\vect{y}_l = \vect{x}_l, \quad \ell=1, \dots L. 
\end{align}
Here $\vect{x}_l \in \C^{MN}$ denotes the $l$th block of the vector $\tvect{x} \in \C^{LMN}$. 
One way to approach this problem is to define $\vect{G}_l := d_l\vect{H}_M + \vect{I}_{M}$ for each $l=1, \dots, L$. 
Then, $\vect{G}_l$ is an upper triangular matrix of size $M \times M$ and it is nonsingular for $\alpha \neq 1$. 
Because of this, the linear systems in~\eqref{eq:inner1} can again be rewritten in two steps as
\begin{subequations}
\begin{align}
\big(\vect{I}_{MN} - \Delta T (\vect{Q}\vect{G}_l^{-1}) \otimes \vect{A}\big)\vect{z}_l &= \vect{x}_l, \label{eq:innercoll}\\
(\vect{G}_l \otimes \vect{I}_N)\vect{y}_l &= \vect{z}_l. \label{eq:Gsubstitution}
\end{align}
\end{subequations}
Now, suppose that $\vect{Q}\vect{G}_l^{-1}$ is a diagonalizable matrix for a given $l$ and $\alpha$. 
The exact circumstances under which this is possible are discussed in section~\ref{sec:diagonlizationofQGinv}. 
If so, one can solve~\eqref{eq:innercoll} by diagonalizing a much smaller $M \times M$ matrix and inverting the matrices on the diagonal in parallel, but now across the collocation nodes (or, as cast in~\cite{Burrage1997}, ``across the method'').
More precisely, let $\vect{Q}\vect{G}_l^{-1} = \vect{S}_l \vect{D}_l \vect{S}_l^{-1}$ denote the diagonal factorization, where $\vect{D}_l = \operatorname{diag}(d_{l1}, \dots, d_{lM})$. 
Hence, the inner systems in~\eqref{eq:innercoll} can now be solved in three steps for each $l = 1, \dots, L$ as
\begin{subequations} \label{eq:innerdiag}
\begin{align}
\big( \vect{S}_l \otimes \vect{I}_N\big)\vect{x}_l^{1} &= \vect{x}_l, \label{eq:d1}\\
\big(\vect{I}_{N} - d_{lm} \Delta T \vect{A} \big) \vect{x}_{lm}^{2} &= \vect{x}_{lm}^{1}, \quad m = 1, \dots, M, \label{eq:d2}\\
\big( \vect{S}_l^{-1} \otimes \vect{I}_N\big)\vect{z}_l &= \vect{x}_l^{2}, \label{eq:d3}
\end{align}
\end{subequations}
where $\vect{x}_{lm} \in \C^{N}$ denotes the $m$th block of $\vect{x}_{l} \in \C^{NM}$.

A summary of the full method is presented as a pseudo-code in algorithm~\ref{alg:paralpha}.
It traces precisely the steps described above and indicates which steps can be run in parallel.
We will describe how a sequence of preconditioners, $(\mathbf{C}_{\alpha_k})_{k \in \N}$, can be prescribed, as well as a stopping criterion in Section~\ref{sec:parameter_selection}.

\begin{algorithm}[t]
\caption{Iterations with a given sequence of $(\alpha_k)_{k \in \N}$.}
\label{alg:paralpha}
\hspace*{\algorithmicindent} \textbf{Input:} $(\vect{C}_{\alpha_k})_{k \in \N}, \,  \vect{C}, \, \tvect{w},  \, \tvect{u}^{(0)}, \, \text{tol}$ \\
\hspace*{\algorithmicindent} \textbf{Output:}  $\text{a solution to } \vect{C}\tvect{u} = \tvect{w}$.
\begin{algorithmic}[1]
\STATE{$k = 0$}
\WHILE{not done}
    \STATE{$\vect{D}_{\alpha_k} = \operatorname{diag}(d_1, \dots, d_L), \, d_l = -\alpha_k^{\frac{1}{L}}e^{-2\pi i\frac{l-1}{L}}$}
    \STATE\label{alg:r}{$\tvect{r} = (\vect{C}_{\alpha_k} - \vect{C})\tvect{u}^{(k)} + \tvect{w}$} 
    \COMMENT{compute the right-hand side}
    \STATE{$\widetilde{\tvect{r}} = (\vect{J}^{-1} \otimes \vect{I}_{MN})\tvect{r}$}
    \STATE\label{alg:fft}{$\tvect{x} = \textbf{FFT}(\widetilde{\tvect{r}})$}
    \hfill\COMMENT{perform the parallel FFT}
    \FOR[{solve on $L$ parallel steps}]{$l = 1, \dots, L$ \textbf{in parallel}}
        \STATE{$\vect{G}_l = d_l\vect{H}_M + \vect{I}_{M}$}
        \STATE{$\vect{Q}\vect{G}_l^{-1} = \vect{S}_l \vect{D}_l \vect{S}_l^{-1}$, $\vect{D}_l = \operatorname{diag}(d_{l1}, \dots, d_{lm})$}
        \STATE\label{alg:d1}{$\vect{x}_l^{1} = \big( \vect{S}_l^{-1} \otimes \vect{I}_N\big)\vect{x}_l$}
            \FOR[{solve on $M$ parallel nodes}]{$m = 1, \dots, M$ \textbf{in parallel}}
                \STATE\label{alg:d2}{\textbf{solve} $\big(\vect{I}_{N} - d_{lm} \Delta T \vect{A} \big) 
                \vect{x}_{lm}^{2} = \vect{x}_{lm}^{1}$}
            \ENDFOR
        \STATE\label{alg:d3}{$\vect{z}_l = \big( \vect{S}_l \otimes \vect{I}_N\big)\vect{x}_l^{2}$}
        \STATE\label{alg:G}{$\vect{y}_l = (\vect{G}_l^{-1} \otimes \vect{I}_N)\vect{z}_l$}
    \ENDFOR
    \STATE\label{alg:ifft}{$\widetilde{\tvect{u}}^{(k+1)} = \textbf{IFFT}(\tvect{y}) $}
    \hfill\COMMENT{perform the parallel IFFT}
    \STATE{$\tvect{u}^{(k+1)} = (\vect{J} \otimes \vect{I}_{MN})\widetilde{\tvect{u}}^{(k+1)}$}
    \hfill\COMMENT{get the new iterate}
    \STATE{$k = k+1$}
\ENDWHILE
\RETURN $\tvect{u}^{(k)}$
\end{algorithmic}
\end{algorithm}

%% file: section2a.tex
\subsection{Diagonalization of \texorpdfstring{$\vect{Q}\vect{G}_l^{-1}$} {}}\label{sec:diagonlizationofQGinv}

In order to solve~\eqref{eq:innercoll} via diagonalization, we need to study when $\vect{Q}\vect{G}_l^{-1}$ is diagonalizable. 
A sufficient condition is if the eigenvalues of $\vect{Q}\vect{G}_l^{-1}$ are distinct.
Here, we will show that for our matrix $\vect{Q}\vect{G}_l^{-1}$ this is also a necessary condition. 

Let $\vect{y}\in \C^M$, then
\begin{equation}\label{eq:Qonvec}
    [\vect{Qy}]_m = \int_0^{t_m}\sum_{i = 1}^M \vect{y}_i c_i(s)ds,  
\end{equation}
where $c_i$ are the Lagrange interpolation polynomials defined using the collocation nodes.
Recall that $\vect{G}_l = d_l\vect{H}_M + \vect{I}_{M}$ is an upper triangular matrix (see~\eqref{eq:Hm}), with a computable inverse that is $\vect{G}_l^{-1} = \vect{I}_M - r_l\vect{H}_M$, where $r_l = d_l / (1+d_l)$. 
We also know that $\vect{Q}\vect{H}_M = \vect{D}_t\vect{H}_M$, where $\vect{D}_t = \operatorname{diag}(t_1, \dots, t_M)$. 
In conclusion, we can write 
\begin{align}\label{eq:QGL}
\vect{Q}\vect{G}_l^{-1} = \vect{Q} - r_l\vect{Q}\vect{H}_M = \vect{Q} - r_l\vect{D}_t\vect{H}_M.
\end{align}
To formulate the problem, we need to find $\lambda \in \C$ and $\vect{y} \neq 0$ such that $\vect{Q}\vect{G}_l^{-1} \vect{y} = \lambda \vect{y}$. 
Because of the reformulation of $\vect{Q}\vect{G}_\ell^{-1}$ in~\eqref{eq:QGL} and~\eqref{eq:Qonvec}, the eigenproblem can now be reformulated as finding a polynomial $h$ of degree $M-1$ such that
\begin{align*}
\int_0^{t_m} h(s)ds + r_l t_m h(t_M) = \lambda h(t_m), \quad m = 1, \dots, M.
\end{align*}
Substituting $h$ with a derivative $g'$, where the polynomial $g$ is of degree $M$, yields
\begin{align}\label{eq:poli}
g(t_m) - g(0) + r_l t_m g'(t_M)= \lambda g'(t_m), \quad m = 1, \dots, M.
\end{align}
From here we see that the eigenvector represented as a polynomial $g$ is not uniquely defined, because if $g$ is a solution of~\eqref{eq:poli}, then $a_M g(t) + a_0$ is also a solution. 
Without loss of generality, we set $g(0) = 0$ so that $g$ takes the form
\begin{align*}
g(t) = t^M + a_{M-1}t^{M-1} + \dots + a_1t.
\end{align*}
With this assumption, $g$ is uniquely defined since we have $M$ nonlinear equations~\eqref{eq:poli} and $M$ unknowns $(a_{M-1}, \dots, a_1, \lambda)$. 
This is equivalent to the statement that an eigenproblem with distinct eigenvalues has a unique eigenvector if it is a unitary vector. 
This reformulation of the problem leads us to the following lemma.

\begin{lemma}\label{lemma:eigens}
Define 
\begin{align*}
w_M(t) = (t - t_M)\dots (t-t_1) = t^M + b_{M-1}t^{M-1} + \dots + b_0.
\end{align*}
Then, the eigenvalues of $\vect{Q}\vect{G}_l^{-1}$ are the roots of
\begin{align*}
p_M(\lambda) = M!\lambda^M + c_{M-1}\lambda^{M-1} + \dots + c_0,    
\end{align*}
where the coefficients are defined as $c_0 = (r_l + 1)b_0$,
\begin{align*}
c_m = m! \, b_m - r_l \sum_{j = 1}^{M-m} \frac{(m+j)!}{j!}b_{m+j}, \quad 1 \leq m \leq M-1
\end{align*}
and $b_M=1$.
\end{lemma}
\begin{proof}
Without loss of generality let $t_M = 1$. 
The eigenproblem is equivalent to solving $M$ nonlinear equations
\begin{align} \label{eq:g1}
g(t_m) + r_l t_m g'(1) - \lambda g'(t_m) = 0, \quad m = 1, \dots, M,
\end{align}
where $g(t) = t^M + a_{M-1}t^{M-1} + \dots + a_1t$ and $\lambda \in \C$. 
Now we define
\begin{align*}
G(t) = g(t) + r_l t g'(1) - \lambda g'(t).
\end{align*}
It holds that $G(t_m) = w_M(t_m)$ for $m = 1, \dots M$ and that the difference $G - w_M$ is a polynomial of degree at most $M-1$, since both $G$ and $w_M$ are monic polynomials of degree $M$.
Because $G - w_M$ is zero in $M$ different points, we conclude that $G = w_M$. 
Equations~\eqref{eq:g1} can compactly be rewritten as
\begin{align}\label{eq:compactpoli}
g(t) + r_l t g'(1) - \lambda g'(t) = w_M(t).
\end{align}
Since the polynomial coefficients on the left-hand side of~\eqref{eq:compactpoli} are the same as on the right-hand side, we get
\begin{subequations}
\begin{align}
-\lambda a_1 &= b_0 \tag{*0} \label{eq:*0} \\
a_1 + r_l g'(1) - 2 \lambda a_2 &= b_1 \tag{*1} \label{eq:*1} \\
a_2 - 3 \lambda a_3 &= b_2 \tag{*2} \label{eq:*2}\\
 &\vdots \notag\\
a_{M-2} - (M-1)\lambda a_{M-1} &= b_{M-2} \tag{*M-2} \label{eq:*M-2}\\ 
a_{M-1} - M\lambda &= b_{M-1}. \tag{*M-1} \label{eq:*M-1}
\end{align}
\end{subequations}
Telescoping these equations starting from~\eqref{eq:*M-1} up to~(*m) for $m \geq 2$, we get
\begin{align*}
a_m = \sum_{j=0}^{M-m} \lambda^j \frac{(m+j)!}{m!} b_{m + j},
\end{align*}
and from~\eqref{eq:*0} we get $a_1 = -b_0/\lambda$. Note that $\lambda \neq 0$ since both $\vect{Q}$ and $\vect{G}_l$ are nonsingular. 
The fact that $\vect{Q}$ is a nonsingular matrix is visible because it is a mapping in a fashion
\begin{align}\label{eq:vandermonde} 
(t_1^m, \dots, t_M^m) \rightarrow \frac{1}{m+1} (t_1^{m+1}, \dots, t_M^{m+1}), \quad 0 \leq m \leq M-1.
\end{align}
The vectors in~\eqref{eq:vandermonde} form columns of the Vandermonde matrix $\vect{W}$ which is known to be nonsingular when $0 < t_1 < \dots < t_M$. 
Because of~\eqref{eq:vandermonde}, we have $\vect{Q}\vect{W} = \vect{D}_M\vect{W}$, where $\vect{D}_M= \operatorname{diag}(1, 1/2, \dots, 1/M)$. 
Now, from here we see that $\vect{Q}$ is nonsingular as a product of nonsingular matrices. 

Substituting the expression for $a_2$ and $a_1$ into~\eqref{eq:*1} and multiplying it with $-\lambda \neq 0$ yields
\begin{align}\label{eq:pp2}
-\lambda r_lg'(1) +  \sum_{m=0}^{M} m! \, \lambda^{m}b_{m} = 0.
\end{align}
It remains to compute $g'(1)$. We have
\begin{align*}
g'(1) &= M + (M-1)a_{M-1} + \dots + 2a_2 + a_1 \\
&= M + \sum_{m=2}^{M-1} m \sum_{j=0}^{M-m} \lambda^j \frac{(m+j)!}{m!} b_{m + j} - \frac{b_0}{\lambda} ,
\end{align*}
whereas
\begin{align*}
\lambda g'(1) = \sum_{m=2}^{M} \sum_{j=0}^{M-m} \lambda^{j+1} \frac{(m+j)!}{(m-1)!} b_{m + j} - b_0.    
\end{align*}
Now we have to shift the summations by 1 and reorder the summation
\begin{align}\label{eq:pp3}
\lambda g'(1) + b_0 = \sum_{m=1}^{M-1} \sum_{j=1}^{M-m} \lambda^j \frac{(m+j)!}{m!} b_{m + j} = \sum_{j=1}^{M-1} \lambda^j \sum_{m=1}^{M-j} \frac{(m+j)!}{m!} b_{m + j}. 
\end{align}
Combining~\eqref{eq:pp2} and~\eqref{eq:pp3} gives a polynomial in the $\lambda$ variable with coefficients being exactly as stated in the lemma.
\end{proof}

\begin{remark}
Because of equation~\eqref{eq:QGL}, lemma~\ref{lemma:eigens} also gives a scaled characteristic polynomial $p_M$ of an arbitrary implicit Runge-Kutta matrix $\vect{Q}$ in the special case when $r_l = 0$.
\end{remark}

From the proof of lemma~\ref{lemma:eigens}, it is visible that each $\lambda$ generates exactly one polynomial $g$ representing an eigenvector, therefore the matrix $\vect{Q}\vect{G}_l^{-1}$ is diagonalizable if and only if the eigenvalues are distinct. 

Lemma~\ref{lemma:eigens} also provides an analytic way of pinpointing the values $r_l$ where $\vect{Q}\vect{G}_l^{-1}$ is not diagonalizable. 
This can be done via obtaining the polynomial of roots for the Gauss-Radau quadrature as $R_M = P_{M-1} + P_M$, where $P_M$ is the $M$th Legendre polynomial which can be obtained with a recursive formula, see~\cite{numODE} for details.
Since $R_M$ has roots in the interval $[-1, 1]$ with the left point included, a linear substitution is required of form $x(t) = -2t + 1$ to obtain the corresponding collocation nodes for the Gauss-Radau quadrature on $[0, 1]$. 
Then the monic polynomial colinear to $(R_M \circ x)(t)$ is exactly $w_M(t)$. 

Now it remains to find the values of $r_l$ for which $p_M$ defined in lemma~\ref{lemma:eigens} has distinct roots.
This can be done  using the discriminant of the polynomial and computing these finitely many values since the discriminant of a polynomial $\Delta_{p_M}$ is nonzero if and only if the roots are distinct. 
The discriminant of a polynomial is defined as $\Delta_{p_M} = (-1)^{M(M-1)/2}/c_M\operatorname{Res}(p_M, p_M')$, and $\operatorname{Res}$ is the residual between two polynomials.
Alternatively, it is often easier to numerically compute the roots of $p_M$ for a given $\alpha$.
For $M = 2, 3$, we give the analysis here.
Note that for $M=1$, the diagonalization is trivial since the matrix $\vect{Q}$ is of size $1 \times 1$.

\begin{example} \label{ex:M2}
When $M=2$, the corresponding polynomial is $w_2(t) = t^2-\frac{4}{3}t + \frac{1}{3}$, generating $p_2(\lambda) = 2\lambda^2 - (\frac{4}{3} + 2r)\lambda + \frac{r + 1}{3}$. 
Solving the equation $\Delta_{p_2} = 0$ is equivalent to $9r^2+6r-2=0$ and the solutions are $r_* = \frac{-1 \pm 3\sqrt{3}}{13}$. 
From here, we can compute $\alpha_*$ which could generate these values.
Then, the matrix $\vect{Q}\vect{G}_l^{-1}$ is not diagonalizable for $\alpha_* \approx 0.323^L, \, 0.477^L$ and these values should be avoided. 
These values are already of order $10^{-8}$ for $L \approx 25$ which is something that should be avoided anyways since $\alpha$ this small tends to generate a large round-off error for the outer diagonalization.
\end{example}
\begin{example}\label{ex:M3}
When $M=3$, the corresponding polynomial is $w_3(t) = t^3 - \frac{9}{5}t^2 + \frac{9}{10}t-\frac{1}{10}$, generating $p_3(\lambda) = 6 \lambda^3 - (\frac{18}{5} + 6r) \lambda^2 + (\frac{9}{10} + \frac{3}{5}r)\lambda - \frac{r+1}{10}$. 
Solving the equation $\Delta_{p_3} = 0$ is equivalent to $1700 r^4 + 3560 r^3 + 1872 r^2 + 18 r + 9 = 0$ and the solutions are $r_* \approx -1.0678, -1.0259, -0.000214 \pm 0.069518 i$.
This generates $\alpha_* \approx 0.516^ L, \, 0.504^L, \, 0.069^L$, for which the diagonalization of $\vect{Q}\vect{G}_l^{-1}$ is not possible.
These alphas are already very small for $L \approx 25$ and should not be used anyways.
\end{example}
 
Examples~\ref{ex:M2} and~\ref{ex:M3} show that for $M=2,3$ and a sufficient number of time-steps $L$, the diagonalization of $\vect{Q}\vect{G}_l^{-1}$ is possible for every time-step $[T_{l-1}, T_l], \, l=1, \dots, L$. 
We also identified that $ \alpha$ around which the diagonalization of the inner method is not possible.

In the next section, we now derive a way to select the parameter $\alpha$ for each iteration automatically.

%% file: section3.tex
\section{Parameter selection}\label{sec:parameter_selection}
After each outer iteration in Algorithm~\ref{alg:paralpha}, a numerical error is introduced. 
Let $\Delta \tvect{u}^{(k+1)}$ denote the error arising after performing the $(k+1)$th iteration for input $\tvect{u}^{(k)}$.
We seek $\alpha$ for each iteration so that the errors and the convergence rate of the method are balanced out.
Our $(k+1)$th error to the solution $\tvect{u}^*$ can be expressed and bounded as
\begin{align}\label{eq:errorstriangle}
    \|\tvect{u}^{(k+1)} + \Delta \tvect{u}^{(k+1)} - \tvect{u}^*\| \leq c_\alpha\|\tvect{u}^{(k)} - \tvect{u}^*\| + \|\Delta \tvect{u}^{(k+1)}\|,
\end{align}
for some constant $c_\alpha > 0$.
The constant $c_\alpha$ is the contraction rate and satisfies the inequality $\|\tvect{u}^{(k+1)} - \tvect{u}^*\| \leq c_\alpha\|\tvect{u}^{(k)} - \tvect{u}^*\|$ on unperturbed values.
Moreover, we expect $\|\Delta \tvect{u}^{(k+1)}\| \rightarrow \infty$ and $c_\alpha \rightarrow 0$ when $\alpha \rightarrow 0$, making the decision on which parameter $\alpha$ to choose for the next iteration ambiguous, but highly relevant for the convergence of the method. 
Bounding the second $\alpha$-dependent term on the right-side of~\eqref{eq:errorstriangle} and approximating $c_\alpha$ will bring us closer in finding suitable $\alpha_{k+1}$ to use for the computation of $\tvect{u}^{(k+1)}$.

\subsection{Convergence}
In order to formulate this problem more precisely, we will utilize the error analysis from~\cite{slw2021convergence}, which we restate here for completeness.

\begin{theorem}\label{theorem:convergencetheorem}
Assume that the matrix $\vect{A} \in \C^{N \times N}$ is diagonalizable as $\vect{A} = \vect{V}_A \vect{D}_A \vect{V}_A^{-1}$ and define $\vect{W} = \vect{I}_L \otimes \vect{V}_A$.
Let $\vect{u}^*$ denote the solution of the composite collocation problem~\eqref{eq:globalsys}.
Then for any $k \geq 1$ it holds
\begin{align*}
    \|\tvect{u}^{(k+1)} - \tvect{u}^*\|_{\vect{W}, \infty} \leq \frac{\alpha}{1-\alpha}\|\tvect{u}^{(k)} - \tvect{u}^*\|_{\vect{W}, \infty}\, ,
\end{align*}
provided that the time-integrator~\eqref{eq:collocation} satisfies $\rho(\vect{C}_\coll) < 1$ and $\|\tvect{u}\|_{\vect{W}, \infty} = \|\vect{W}\tvect{u}\|_{\infty}$.
\end{theorem}
\begin{proof}
See~\cite{slw2021convergence}.
\end{proof}
Using the above result, we can write
\begin{align*}
\|\tvect{u}^{(k+1)} - \tvect{u}^*\|_\infty &\leq \|\vect{W}^{-1}\|_\infty \|\vect{W}(\tvect{u}^{(k+1)} - \tvect{u}^*)\|_\infty \\
& \leq \|\vect{W}^{-1}\| \frac{\alpha}{1-\alpha}\|\vect{W}(\tvect{u}^{(k)} - \tvect{u}^*)\|_\infty \\
& \leq \kappa_\infty(\vect{W}) \frac{\alpha}{1-\alpha} \|\tvect{u}^{(k)} - \tvect{u}^*\|_\infty.
\end{align*}
This then leaves the question of how to bound the second error term $\|\Delta \tvect{u}^{(k+1)}\|$ in equation~\eqref{eq:errorstriangle}.


\subsection{Bounds for computation errors}\label{sec:bounds}
Alongside the round-off errors, the existing bound assumes exact system solves, which, in reality, is never the case \cite{wr2019}. 
Because the following error bounds are very general and can be applied to any diagonalization-based algorithm, we state them in simplified notation.
A similar analysis has been done in~\cite{round_off}, however, it was performed on triangular matrices and assuming exact system solves.

The three steps in the diagonalization computation~\eqref{eq:s123} can be generally analyzed in this order: a matrix-vector multiplication $y = \tA x$, solving a system $\tB z = y$, where $\tB$ is a block diagonal matrix, followed by a matrix-vector multiplication $w = \tC z$, where in our particular case $\tC$ is the inverse of $\tA$.
The errors in each step can be expressed as
\begin{subequations}\label{eq:diaggeneric}
\begin{gather}
y + \Delta y = (\tA + \Delta \tA)x \\
(\tB + \Delta \tB) (z + \Delta z) \approx y + \Delta y \label{eq:es2} \\
w + \Delta w = (\tC + \Delta \tC)(z + \Delta z)
\end{gather}
\end{subequations}
with relative errors satisfying 
\begin{align}\label{eq:relerrors}
\|\Delta \tA\| \leq \varepsilon \|\tA\|, \, \|\Delta \tB\| \leq \varepsilon \|\tB\|, \, \|\Delta \tC\| \leq \varepsilon \|\tC\|,
\end{align}
for some $\varepsilon > 0$, representing machine precision.

\begin{lemma}\label{lemma:ultimate}
Let the system solve in~\eqref{eq:es2} be inexact, in other words $\|(\tB + \Delta \tB) (z + \Delta z) - (y + \Delta y)\| \leq \tau \|y + \Delta y\|$ for some $\tau>0$ and let~\eqref{eq:relerrors} hold. 
Then the norm of the absolute error of $w$ after the diagonalization process~\eqref{eq:diaggeneric} satisfies 
\begin{align*}
\|\Delta w\| \leq \frac{\|\tB^{-1}\|\|\tA\|\|\tC\|}{1 - \varepsilon \kappa(\tB)}(2 \varepsilon + \tau + \varepsilon \kappa(\tB))\|x\| + O(\tau \varepsilon + \varepsilon^2),
\end{align*}
where the matrix norm is consistent, i.e.~$\|\tA x\| \leq \|\tA\|\|x\|$, and the condition number is $\kappa(\tA) := \|\tA\|\|\tA^{-1}\|$.
\end{lemma}
\begin{proof}
See appendix~\ref{app:proof}.
\end{proof}

\begin{theorem}\label{theorem:errortheorem}
After one iteration of~\eqref{eq:Riteration}, for $\alpha \in (0, 1)$, under the assumption that the inner system solves in step~\eqref{eq:s2} satisfy $\|\tB (z + \Delta z) - (y + \Delta y)\|_\infty \leq \tau \|y + \Delta y\|$, the error of the $(k+1)th$ iterate can be bounded by
\begin{align}\label{eq:errr1}
\|\Delta \tvect{u}^{(k+1)}\|_\infty \leq \frac{\|\tB^{-1}\|_\infty}{1 - \varepsilon \kappa_\infty(\tB)} \frac{L(2\varepsilon + \tau + \varepsilon \kappa_\infty(\tB))}{\alpha}\|\tvect{r}^{(k)}\|_\infty + O(\varepsilon \tau + \varepsilon^2),
\end{align}
where $\varepsilon$ is machine precision as in~\eqref{eq:relerrors}, $\vect{r}^{(k)} = (\vect{C}_\alpha - \vect{C})\tvect{u}^{(k)} + \tvect{w}$ and $\tB = \vect{D}_\alpha \otimes \vect{H} + \vect{I}_L \otimes \vect{C}_\coll$. 
Furthermore, for $l = 1, \dots, L$ we have
\begin{align}\label{eq:errr2}
\|\Delta \vect{u}^{(k+1)}_l\|_\infty \leq \frac{\|\tB^{-1}\|_\infty}{1 - \varepsilon \kappa_\infty(\tB)} \alpha^{\frac{-(l-1)}{L}}L(2\varepsilon + \tau + \varepsilon \kappa_\infty(\tB))\|\tvect{r}^{(k)}\|_\infty + O(\varepsilon \tau + \varepsilon^2).
\end{align}
\end{theorem}

\begin{proof}
Let relation~\eqref{eq:relerrors} hold.
Define $\tA := \vect{V}^{-1} \otimes \vect{I}_{MN}$, $\tB = \vect{D}_\alpha \otimes \vect{H} + \vect{I}_L \otimes \vect{C}_\coll$ and $\tC := \vect{V} \otimes \vect{I}_{MN}$ where without loss of generality we can assume $\tA = \vect{V}^{-1}$ and $\tC = \vect{V}$. From lemma~\ref{lemma:diag} we see that 
\begin{align*}
    \|\tA\|_\infty\|\tC\|_\infty \leq \frac{1}{L} \|\vect{J}\|_\infty \|\vect{J}^{-1}\|_\infty \|\vect{F}\|_\infty \|\vect{F}^*\|_\infty \leq \frac{L}{\alpha}
\end{align*}
since $\|\vect{F}\|_\infty \leq L$ and $\|\vect{F}^*\|_\infty \leq L$.
The proof for~\eqref{eq:errr1} now follows directly from lemma~\ref{lemma:ultimate}.
Furthermore, if we define $\tA = \vect{V}^{-1}$ and $\tC = \frac{1}{L}\vect{F}$, then
\begin{align*}
\|\tA\|_\infty\|\tC\|_\infty \leq \frac{1}{L} \|\vect{J}^{-1}\|_\infty \|\vect{F}\|_\infty\|\vect{F}^*\| \leq L.
\end{align*}
If now we define $\Delta w = (\vect{J}^{-1} \otimes \vect{I}_{MN})\Delta \tvect{u}^{(k+1)}$, we have
\begin{align*}
\|\vect{u}^{(k+1)}_l\|_\infty =  \alpha^{\frac{-(l-1)}{L}}\|\Delta w_l\|_\infty \leq \alpha^{\frac{-(l-1)}{L}}\|\Delta w\|_\infty
\end{align*}
and the proof for~\eqref{eq:errr2} again follows directly from lemma~\ref{lemma:ultimate}.
\end{proof}

\begin{remark}
The IEEE standard guarantees that relation~\eqref{eq:relerrors} holds for the infinity norm and some $\varepsilon = 2^{-p}$ representing machine precision.
\end{remark}

The above theorem provides an idea of how the round-off error propagates across the vector $\tvect{u}^{(k+1)}$. 
The larger the number of the time-steps, the larger the round-off error we can expect. 
This means that if the error between consecutive iterates is monitored, it is sufficient enough to do so just on the last time-step.
Thus, we can avoid computing a residual and use the difference between two consecutive iterates at the last time-step as a termination criterion.

\subsection{Choosing the \texorpdfstring{$(\alpha_k)_{k \in \N}$}{} sequence} \label{sec:optimalparameterchoice}
While theorem~\ref{theorem:convergencetheorem} tells us that a small $\alpha$ yields fast convergence, theorem~\ref{theorem:errortheorem} indicates that a smaller $\alpha$ leads to a larger numerical error per iteration. 
Suitable choices of $\alpha$ should balance both aspects.
To achieve that, define $m_0$ so that $\|\tvect{u}^{(0)} - \tvect{u}_*\|_\infty \approx m_0$. 
Using the results of theorems~\ref{theorem:convergencetheorem} and~\ref{theorem:errortheorem} we can approximate 
\begin{align*}
    \|\tvect{u}^{(1)} - \tvect{u}^*\|_\infty \lessapprox \alpha m_0
\end{align*}
for $\alpha \approx 0$ and
\begin{align*}
\|\Delta \tvect{u}^{(1)}\|_\infty \lessapprox \frac{\|\tB^{-1}\|_\infty}{1 - \varepsilon \kappa_\infty(\tB)} \frac{L(2\varepsilon + \tau + \varepsilon \kappa_\infty(\tB))}{\alpha}\|\tvect{r}^{(0)}\|_\infty \lessapprox \frac{L}{\alpha}(3\varepsilon + \tau)\|\tvect{r}^{(0)}\|_\infty,
\end{align*}
where, as before, $\tvect{r}^{(0)} = (\vect{C}_\alpha - \vect{C})\tvect{u}^{(0)} + \tvect{w}$. 
The first estimate originates from the fact that $\kappa_\infty(\vect{W})\alpha/(1 - \alpha) = o(\alpha)$.
We will later verify that the convergence rate is unaffected by the simplification (see figure~\ref{fig:2eqconv}).
This assumption does not change the asymptotics for small alphas, but it allows us to draw more conclusions about the trade-off of errors, especially around $\alpha \approx 0$.
The sharpness of the convergence bound from theorem~\ref{theorem:convergencetheorem} is not very relevant, since the convergence factor is not practical to use it in that form.
Combining these estimates, we get
\begin{align}\label{eq:minimeq}
\|\tvect{u}^{(1)} + \Delta \tvect{u}^{(1)} - \tvect{u}^*\|_\infty \lessapprox \alpha m_0 + \frac{L}{\alpha}(3\varepsilon + \tau)\|\tvect{r}^{(0)}\|_\infty.
\end{align}
The aim is to find $\alpha$ such that the right-hand side of~\eqref{eq:minimeq} is minimized for given $L$, $\varepsilon$, and $\tau$.
Since $(\vect{C}_\alpha - \vect{C})\tvect{u}^{(0)}$ is a block-vector with just $-\alpha \vect{H}\vect{u}^{(0)}_L$ being a nonzero block, we can treat the $\|\tvect{r}^{(0)}\|_\infty$ term in the minimization process roughly as $\|\tvect{w}\|_\infty$ since this part is more relevant. 
Defining $\gamma := L(3\varepsilon + \tau)\|\vect{w}\|_\infty$ for simplicity, we can state the problem as: find $\alpha \in \langle 0, 1\rangle$ such that $m_1 = \alpha m_0 + \frac{\gamma}{\alpha}$ is minimized. 
The solution is visible from the fact that $ m_0\alpha^2 - m_1\alpha + \gamma = 0$ is a parabola with a discriminant $\Delta = m_1^2 - 4m_0\gamma$ and that the smallest value for $m_1$ so that $\Delta \geq 0$ is when $\Delta = 0$. 
This yields $m_1 = 2\sqrt{m_0\gamma}$ and a unique root $\alpha = \sqrt{\frac{\gamma}{m_0}}$.
To interpret this result, we expect the error after one iteration to be around $m_1$ if the parameter $\alpha_1 = \sqrt{\frac{\gamma}{m_0}}$ is used. 

Recursively using~\eqref{eq:minimeq}, we can approximate
\begin{align*}
\|\tvect{u}^{(k+1)} + \Delta \tvect{u}^{(k+1)} - \tvect{u}^*\| \lessapprox \alpha m_k + \frac{\gamma}{\alpha}.
\end{align*}
Solving the same minimization problem now for $m_{k+1} = \alpha m_k + \frac{\gamma}{\alpha}$ yields $m_{k+1} = 2\sqrt{m_k\gamma}$ and $\alpha = \alpha_{k+1} = \sqrt{\frac{\gamma}{m_k}}$. 

If $4\gamma \leq m_k$ holds, then the approximations of errors $m_k$ are decreasing. 
Now we can conclude that the sequence of $(\alpha_k)_{k\in \N}$ is increasing. 
Since $(\alpha_k)_{k\in \N}$ is additionally bounded from above with $\sqrt{\gamma / m_k} \leq 1/2$, we see that it is a convergent sequence with a limit in $(0, 1/2]$. 
Telescoping the recursion for $m_k$ we get $m_k = (4\gamma)^{1-2^{-k}}m_0^{2^{-k}}, \, k \geq 1$, showing that the sequence is asymptotically bounded by $4\gamma$. 
The convergence speed seems to be slower as the method iterates, but this does not have to be the case in practice. 
A remedy to this is to monitor the error of consecutive iterates and because of theorem~\ref{theorem:errortheorem} we can do so just for the last time-step. 
If this error is far lower than $m_k$, it is a good sign the convergence is much better than anticipated. 
Usually, monitoring just the error of consecutive iterates does not detect stagnation, but sees it as convergence. 
However, the combination of both can provide solid information. 
This discussion can be summarized in the following algorithm for the stopping criterion.

\begin{algorithm}[ht]
\caption{A stopping criterion combining approximations of worst-case convergence and an error of consecutive iterates.}
\begin{algorithmic}[1]
\STATE{$\gamma := L(3\varepsilon + \tau)\|\tvect{w}\|_\infty$}
\STATE{$k = 0$}
\WHILE{$m_k >\operatorname{tol}$}
    \STATE{$m_{k+1} = 2\sqrt{m_k\gamma}$}
    \STATE{$\alpha_{k+1} = \sqrt{\gamma / m_k}$}
    \STATE{$\tvect{u}^{(k+1)} = $ \textbf{iterate}($\tvect{u}^{(k)}$, $m_{k+1}$, $\alpha_{k+1}$)}
    \IF{$\|\vect{u}_L^{(k+1)} - \vect{u}_L^{(k)}\|_\infty \leq \operatorname{tol}$}
        \RETURN{$\tvect{u}^{(k+1)}$}
    \ENDIF
    \STATE{$k = k + 1$}
\ENDWHILE
\end{algorithmic}\label{alg:alphas}
\end{algorithm}

The only problem remains how to approximate $m_0$. 
Since it is an approximation of $\|\tvect{u}^{(0)} - \tvect{u}^*\|_\infty$ it is convenient to define a vector filled with initial conditions $\tvect{u}_0$ as the initial iterate $\tvect{u}^{(0)}$. 
Then, from the integral formulation, we have $m_0 \approx \max_{[0, T]}|u(t)-u(0)|\leq TL_u$ where $L_u$ is the local Lipschitz constant of the solution on $[0, T]$. 
If this is inconvenient, one can also use that $m_0 \approx \|\tvect{u}^{(0)} - \tvect{u}_*\|_\infty \leq T \max_{[0, T]}\|Au(t) + b(t)\|_\infty \leq T (\|\vect{A}\|_\infty M_u + M_b)$ where $M_u$ denotes the approximation of a maximum of the solution and $M_b$ is the maximum for the function $b$. 
The best approximate convergence curves are when $m_0$ is a tight approximation. 
One can say that the errors are around values $m_k$ as much as $m_0$ is around $\|\tvect{u}^{(0)} - \tvect{u}_*\|_\infty$. 
In practice, if $m_0$ is an upper bound, then all $m_k$ are an upper bound.

This concludes the discussion about the $(\alpha_k)_{k \in \N}$ sequence. 
In this section, we have proposed an $(\alpha_k)_{k\in\mathbb{N}}$ sequence that is supposed to minimize the number of outer iterations, balancing the unwanted numerical errors and the convergence rate.
In the next section, we will discuss the actual implementation of the algorithm on a parallel machine.

%% file: section4.tex
\section{Implementation} \label{sec:implementation}

In this section, an MPI-based parallel implementation is described in detail. 
This is the foundation of all the speedup and efficiency graphs presented in the results section, based on a code written in Python in the framework of \texttt{mpi4py} and \texttt{petsc4py} \cite{py1, py4}. 

\subsection{Parallelization strategy}\label{sec:parallelization}
Let the number of processors be $\nproc = \nstep \ncoll \nspace$, where $\nstep$ groups of processors handle the parallelization across the time-steps and $\ncoll \nspace$ processors handle the parallelization across the method.
Let the number of processors for time-step parallelism be $\nstep = L$, where $L$ is a power of $2$.
This means that each group of $\ncoll \nspace$ processor stores one single approximation $\vect{u}^{(k)}_l\in\mathbb{C}^{MN}$, denoting the corresponding time-step block of $\tvect{u}^{(k)}\in\mathbb{C}^{LMN}$ defined in~\eqref{eq:Riteration}. 
For the parallelization across the method, we set $\ncoll = M$, so that exactly $M$ groups of $\nspace$ processors deal with the collocation problem at each step.
When $\nspace = 1$, each of the $LM$ processors stores a single vector $u_m\in\mathbb{C}^N$ of $\vect{u}^{(k)}_l = (u_1, \dots, u_M)^T$, representing one implicit stage of the collocation problem. 
If $\nspace > 1$, i.e.\ if spatial parallelism is used, then each of the $LM\nspace$ processors only stores a part of these implicit stages.
See figure~\ref{fig:communicators} for examples. 

\begin{figure}[!htb]
\centering
\includegraphics[width=0.8\textwidth]{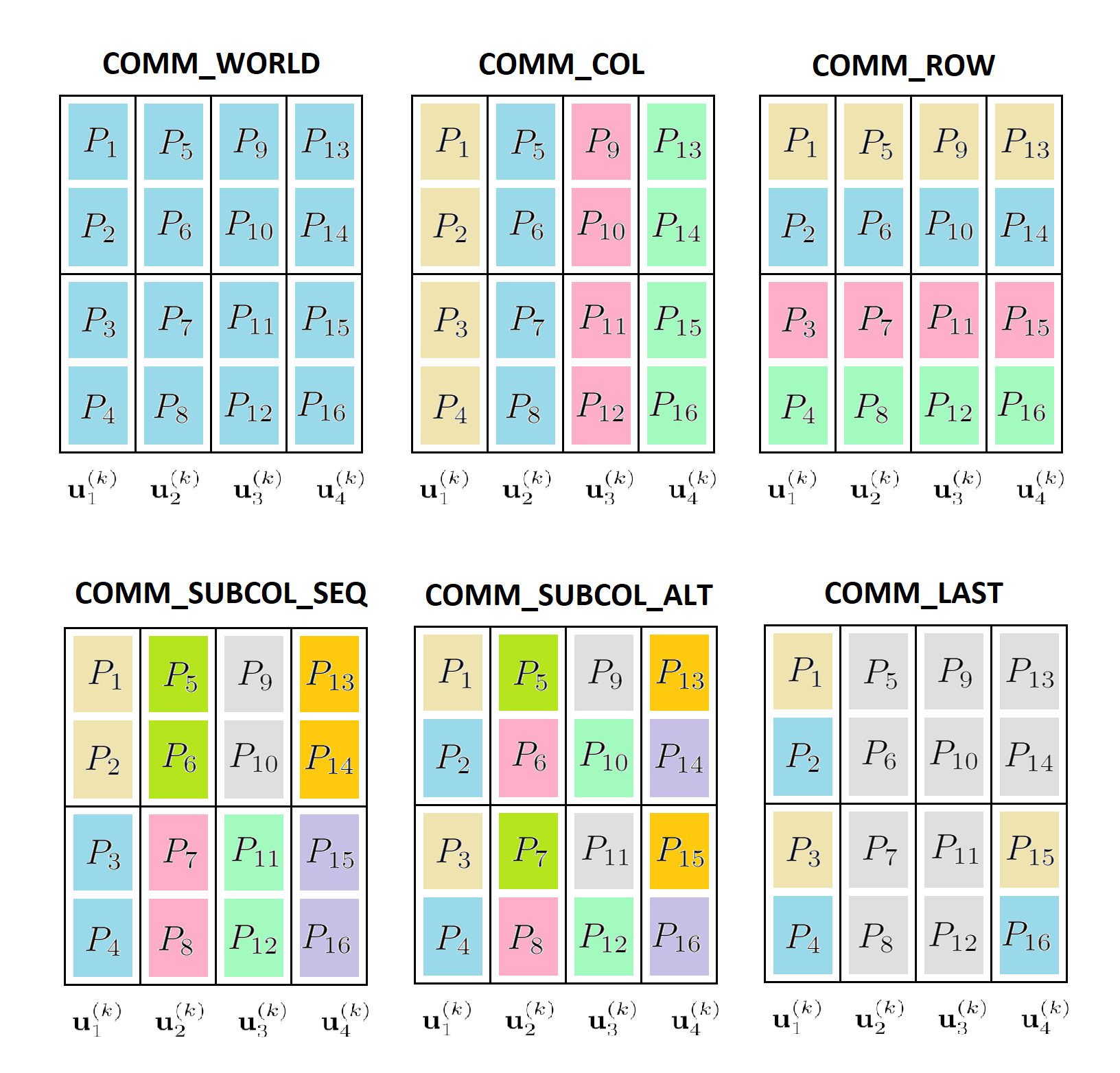}
\caption{This figure represents the communicating groups for $L = 4$ time-steps and the collocation method with $M = 2$ internal stages. There are $\nproc = 16$ processors of what $\nstep = 4$, $\ncoll = 2$ and $\nspace = 2$. Each colored block is a part of the vector locally stored on a processor and the different color groups represent the subcommunicators of the MPI\_COMM\_WORLD. Each column of the table represents the storage for vector $\vect{u}^{(k)}_l$ whereas the upper and the lower parts of the column are for the two implicit stages of the collocation problem. When $\ncoll\nspace = M$, groups COMM\_SUBCOL\_SEQ and COMM\_SUBCOL\_ALT are nonexistent.}
\label{fig:communicators}
\end{figure}

The diagonalization process can be seen as first computing~\eqref{eq:s1} (line~\ref{alg:fft} in algorithm~\ref{alg:paralpha}) as a parallel FFT with a radix-2 algorithm on a scaled vector $(\widetilde{\vect{r}}_1, \dots, \widetilde{\vect{r}}_L)$. 
After that, each group of processors that stored $\widetilde{\vect{r}}_l$ now holds $\vect{x}_{l_i}$, where the index $l_i$ emerges from the butterfly structure which defines the communication scheme of radix-2 (see figure~\ref{fig:radix2}). 
At this stage the problem is decoupled, therefore there is no need to rearrange the vectors back in the original order as done in the standard radix-2 algorithm, which saves communication time. 
The inner system solves in~\eqref{eq:s2} can be carried out on these perturbed indices until the next radix-2 for the parallel IFFT in~\eqref{eq:s3} (line~\ref{alg:ifft} in algorithm~\ref{alg:paralpha}) is performed. 
The butterfly structure for communication is stretched across processors in the COMM\_ROW subgroup and the communication cost is $O\big(2\operatorname{log}_2(L)\big)$ with chunks of memory sent and received being $O(MN/(\ncoll \nspace))$.

\begin{figure}[!htb]
\centering
\includegraphics[width=\textwidth]{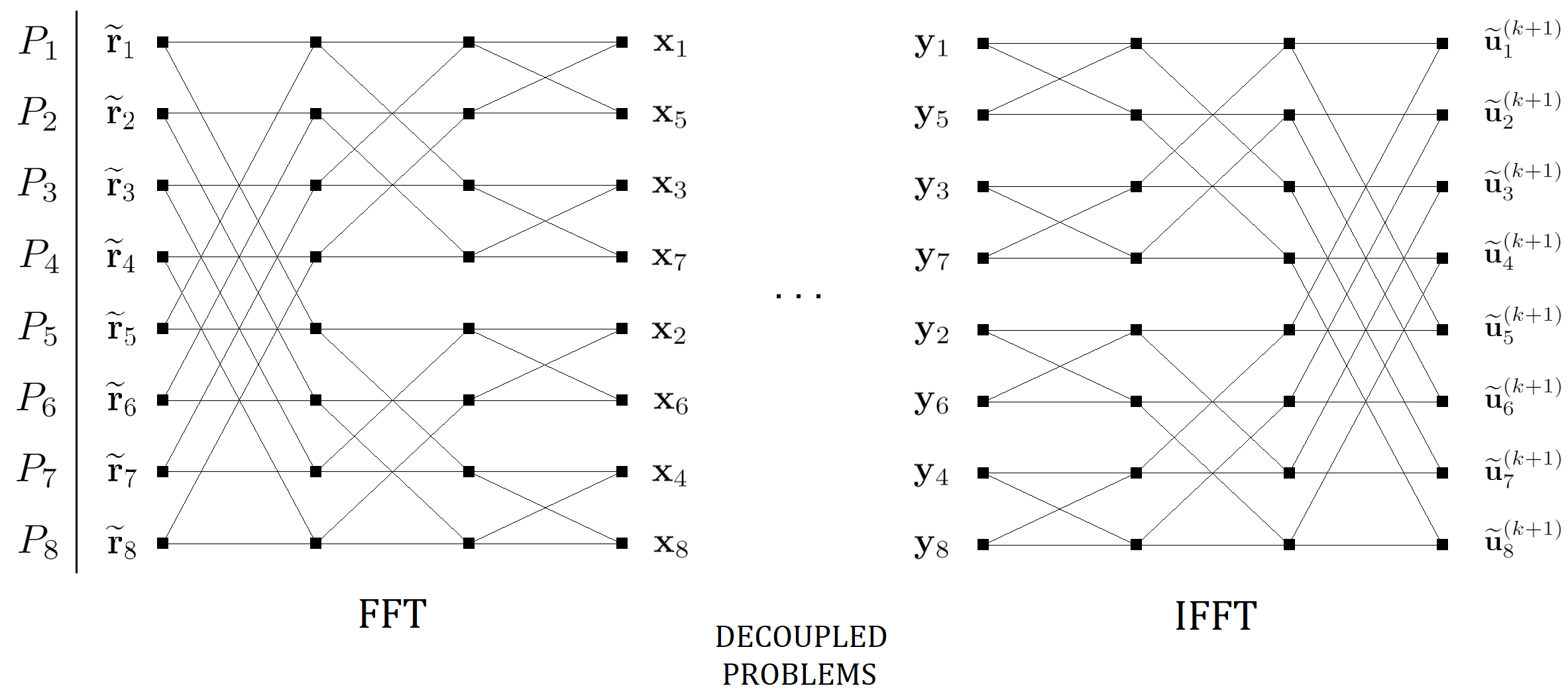}
\caption{An example of a radix-2 butterfly communication structure for $L=\nstep=8$ time-steps and processors. Because the indices after the forward Fourier transform do not need rearranging, the computation proceeds with the perturbed blocks of $\vect{x}_{l_i}$.}
\label{fig:radix2}
\end{figure}

On the other hand, solving the diagonal systems in~\eqref{eq:s2} requires more care. 
Each group and subgroup of COMM\_COL simultaneously and independently solve its own system, storing the solutions in $\vect{y}_l$. 
Excluding the communication time, this part is expected to be the most costly one. 
After forming and diagonalizing the matrix $\vect{Q}\vect{G}_l^{-1}$ locally on each processor, $\vect{x}_l^{1}$ as in~\eqref{eq:d1} (line~\ref{alg:d1} in algorithm~\ref{alg:paralpha}) is firstly computed as
\begin{align*}
(\vect{x}_{l}^{1})_m &= \sum_{j = 1}^M [\vect{S}_l^{-1}]_{mj} (\vect{x}_l)_j, \quad  (\vect{x}_{l}^{1})_m, (\vect{x}_l)_j \in \C^{N}, \, \vect{S}_l\in \C^{M\times M},
\end{align*}
where $(\vect{x}_l)_m$ is a subvector of $\vect{x}_l \in \C^{NM}$ containing indices from $mN$ to $(m+1)N$ which represents the corresponding implicit collocation stage (parts of columns in figure~\ref{fig:communicators} separated with black lines). 
When $\nspace > 1$, this summation can be computed using MPI-reduce on the COMM\_SUBCOL\_ALT level for each $m = 1, \dots, M$ with computational complexity $O(M \log_2(M))$. 
If $\ncoll \leq M$, then the communication complexity is $O(n_\coll \log_2(n_\coll))$ on the COMM\_COL level with chunks of memory sent and received being $O(MN/(\ncoll \nspace))$ in both cases.
Furthermore, equations~\eqref{eq:d3} and~\eqref{eq:Gsubstitution} (lines~\ref{alg:d3} and~\ref{alg:G} in algorithm~\ref{alg:paralpha}) are computed in exactly the same manner as discussed above.

Finally, the inner linear systems after both diagonalizations in~\eqref{eq:d2} (line~\ref{alg:d2} in algorithm~\ref{alg:paralpha}) are solved using GMRES (without a preconditioner).
When $\ncoll \leq M$, each group of $M/\ncoll$ linear systems is solved simultaneously, in a consecutive way within a group, without passing on any additional processors to \texttt{petsc4py}.
If $\nspace > 1$, then we have additional processors for handling the systems of size $N \times N$ in parallel on the COMM\_SUBCOL\_SEQ level which is passed on to \texttt{petsc4py} as a communicator. 
Here in this step, there is a lot of flexibility for defining the actual linear solver since \texttt{petsc4py} is a well-developed library with a lot of different options.
For our purpose, we make use of GMRES~\cite{GMRES} without a preconditioner.

\subsection{Computational complexity and speedup analysis}\label{sec:sande}
In this section, we will explore the theoretical speedup of parallelization in time with our algorithm. 
In order to do that, we need computational complexity estimates for three different cases:
an estimate for a completely sequential implementation $T_\mathrm{seq}$, an estimate if there are $M$ processors available for solving the problem sequentially over time steps and parallel across the collocation points $T_\mathrm{Mpar}$, and lastly, an estimate if we have $LM$ processors to solve in parallel across the steps (lines 6 and 17 in Algorithm \ref{alg:paralpha}) and across the collocation points (lines 11-13 in Algorithm \ref{alg:paralpha}) $T_\mathrm{par}$.
Communication times in this model are ignored, and we assume to handle algebraic operations with the same amount of memory chunks in all cases.
The estimates can be given as:
\begin{align*}
    T_\mathrm{seq} &= L(MT_\mathrm{sol} + 2M^2) = LM(T_\mathrm{sol} + 2M), \\
    T_\mathrm{Mpar} &= L(T_\mathrm{sol} + 2M\log(M)),\\
    T_\mathrm{par} &= k(T_\mathrm{sol,par} + 2\log(L) + 3M\log(M)).
\end{align*}

Quantity $T_\mathrm{seq}$ denotes the computational complexity when solving the collocation problem~\eqref{eq:collocation} sequentially over $L$ time-steps directly via diagonalization of $\vect{Q}$. 
For each step, the solution is obtained by solving each system on the diagonal, one by one, $M$ times, with a solver complexity of $T_\mathrm{sol}$. 
The expression $2M^2$ stands for the two matrix-vector multiplications needed for the diagonalization of $\vect{Q}$.

$T_\mathrm{Mpar}$ denotes the complexity for solving the collocation problem sequentially over $L$ time-steps via diagonalization, except these diagonal systems of complexity $T_\mathrm{sol}$ are solved in parallel across $M$ processors. 
The two matrix-vector products can then be carried out in parallel with a computational complexity of $2M\log(M)$. 

At last,  $T_\mathrm{par}$ represents the complexity of the algorithm with $\nstep = L, \ncoll = M$ processors. 
Here, $k$ denotes the number of outer iterations, and $2\log(L) + 3M\log(M)$ is the complexity of operations when using the communication strategies as discussed in section~\ref{sec:parallelization}.
The solver complexity $T_\mathrm{sol,par}$ may differ from $T_\mathrm{sol}$ since differently conditioned systems may be handled on the diagonals. 
Also, the algorithm requires solving complex-valued problems even for real-valued systems, which could cause an overhead as well, depending on the solver at hand.

Hence, the theoretical speedups for $k\in\mathbb{N}$ iterations look like this:
\begin{subequations}
\begin{align}
    \frac{T_\mathrm{seq}}{T_\mathrm{par}} &= \frac{LM(T_\mathrm{sol} + 2M)}{k(T_\mathrm{sol,par} + 2\log(L) + 3M\log(M))}, \label{eq:speedup1}\\
    \frac{T_\mathrm{Mpar}}{T_\mathrm{par}} &= \frac{L(T_\mathrm{sol} + 2M\log(M))}{k(T_\mathrm{sol,par} + 2\log(L) + 3M\log(M))}. \label{eq:speedup2}
\end{align}
\end{subequations}
The true definition of speedup would be~\eqref{eq:speedup1}, however, the baseline method in our algorithm is also a parallel method (parallel across the method, i.e.~over the collocation nodes), therefore we also need to take~\eqref{eq:speedup2} into account. For $L \geq 2$ we have
\begin{align*}
2\log(L) + 3M\log(M) \geq 2M
\end{align*}
which in combination with~\eqref{eq:speedup1} gives
\begin{align}
\frac{T_\mathrm{seq}}{T_\mathrm{par}} \leq \frac{LM}{k} \frac{T_\mathrm{sol} + 2M}{T_\mathrm{sol,par} + 2M}. \label{eq:speed1}
\end{align}
If we combine the fact that 
\begin{align*}
2\log(L) + 3M\log(M) \geq 2M\log(M)
\end{align*}
is true for all $M, L$ and~\eqref{eq:speedup2}, we get
\begin{align}
\frac{T_\mathrm{Mpar}}{T_\mathrm{par}} \le \frac{L}{k}\frac{T_\mathrm{sol} + 2M\log(M)}{T_\mathrm{sol,par} + 2M\log(M)}. \label{eq:speed2}
\end{align}
We can see that the speedup estimates~\eqref{eq:speed1} and~\eqref{eq:speed2} roughly depend on the ratio between how many steps and nodes are handled in parallel and the number of outer iterations. 
This is not a surprise when dealing with Parareal-based parallel-in-time methods, and our speedup estimates fit the usual theoretical bounds in this field, too. 
Most importantly, as in Parareal, it is crucial to minimize the number of iterations in order to achieve parallel performance. 
This is the reason why carefully choosing the $(\alpha_k)_{k\in\mathbb{N}}$ sequence is important, it brings the number of outer iterations down.
Lowering the outer iteration count for even a few iterations can provide significant speedup.
This claim is not just supported theoretically, but in our test runs as well, as seen in graphs~\ref{fig:strongscaling}.

%% file: section5.tex

\section{Numerical results}

We study the behavior of the algorithm for two different linear test problems. 
The equations were chosen in a way that their analytical solutions are known in order to compute the difference between the approximate solution and the exact one.
The time domain is set to $[T_0, T_0 + T]$, where $T$ is chosen so that the error of the stable time-stepping method satisfies the expected discretization order.
We disabled multithreading which is usually silently triggered by  \texttt{numpy}.
A vector filled with initial conditions $\vect{u}^{(0)} = (\vect{u}_0, \dots, \vect{u}_0)$ is chosen as initial guess for the iteration.
All the results presented here were obtained with our Python implementation, which can be found on GitHub~\cite{paralphacode}, and performed on the supercomputer JUWELS~\cite{JUWELS}.
A direct comparison of performance to other methods such as MGRIT and PFASST is out of the scope of this analysis and is left for future work.
This is the main reason why speedup is presented as a measure, compared to wallclock times which vary much more depending on the machine and the implementation.
Some comparison is made between ParaDiag-II, MGRIT, and Parareal in~\cite{para_DIAG}.
In the preliminary work of~\cite{para_DIAG}, the method used for comparison is backward Euler only, without the parameter adaptivity and without consideration of inexactness in system solves.
The approach, already there, shows promising scaling properties, even outperforming the classical PinT methods.

The first test equation here is the heat equation, governed by
\begin{align}\label{eq:heat}
	u_t = \Delta u + \sin(2 \pi x)\sin(2 \pi y) (8\pi^2 \cos(t) - \sin(t)), \text{ on }[\pi, \pi + T] \times [0, 1]^2,
\end{align}
and the exact solution is $u(t, x, y) = \sin(t)\sin(2 \pi x)\sin(2 \pi y)$. 
This equation has periodic boundary conditions which were used to form the discrete periodic Laplacian with central differences bringing the equation into the generic form~\eqref{eq:genericlinearform}.

The second equation is the advection equation defined as
\begin{align}\label{eq:advection}
	u_t + u_x + u_y = 0, \text{ on }[0, T] \times [0, 1]^2,
\end{align}
with exact solution $u(t, x, y) = \sin(2 \pi x - 2 \pi t)\sin(2 \pi y - 2 \pi t)$. 
Here, we again have periodicity on the boundaries which was used to form an upwind scheme.


\subsection{Counting iterations}
The theoretical background of section~\ref{sec:parameter_selection} is verified for the advection equation.
Two things are highlighted in this section: the importance of $\alpha$-adaptivity and how the convergence curve when using adaptivity compares to individual runs with fixed $\alpha$.
Figure~\ref{fig:advection64} compares the convergence when the adaptive $(\alpha_k)_{k \in \N}$ sequence generated on runtime is used, to runs when $\alpha_k$ is a fixed parameter, from that same generated sequence.
For larger values of $\alpha$, we can see a slow convergence speed that can reach better accuracy while for smaller $\alpha$ values the convergence is extremely steep, but short living.
Using the adaptive $(\alpha_k)_{k \in \N}$ requires fewer total iterations to recover a small error tolerance compared to a fixed $\alpha$.

The numbers of discretization points in space and time are chosen so that the error with respect to the exact solution (infinity norm) is below $10^{-12}$, without over resolving in space or time. 
The parameters chosen are $T = 0.0128$, $L = 64$, $M = 3$, $5$th-order upwind scheme in space with $N = 700$.
The inner solver is GMRES with a relative tolerance $\tau=10^{-15}$.
The test for the heat equation looks very similar and is not shown here.


\begin{figure}[!htb]
\centering
\includegraphics[width=0.75\textwidth]{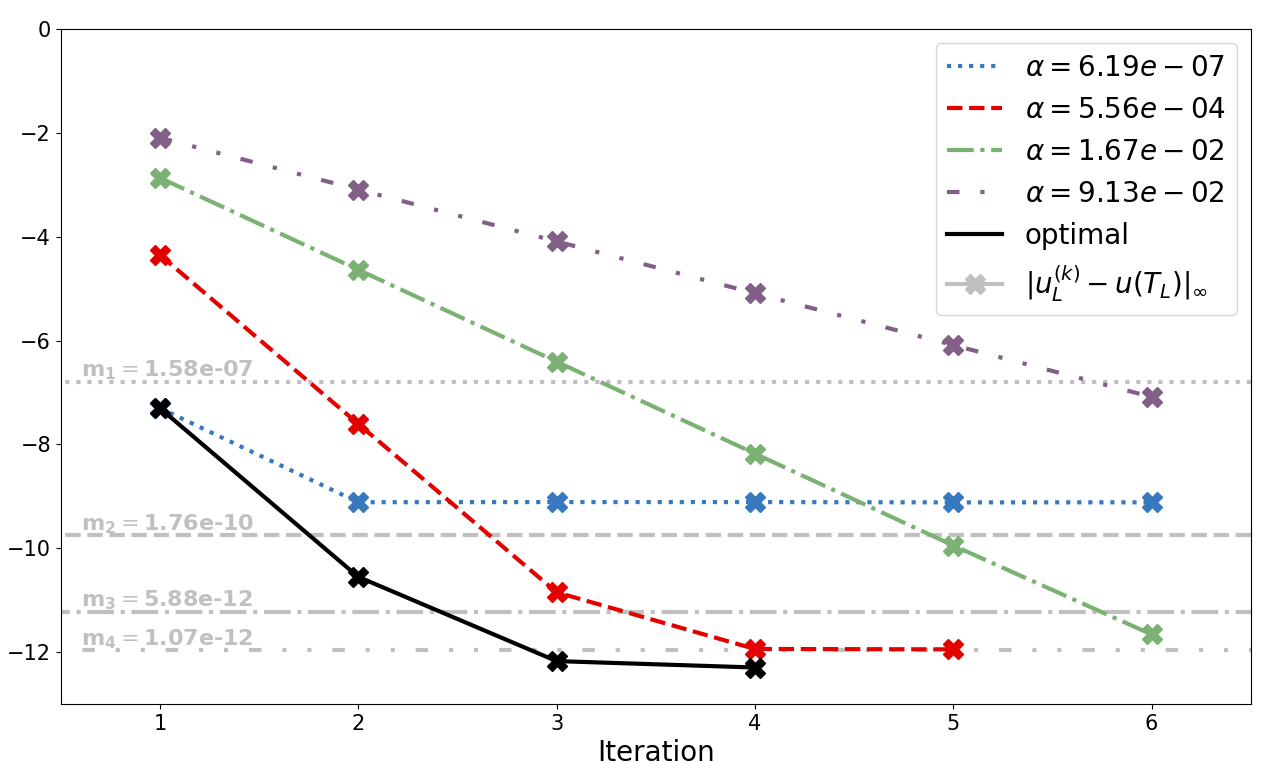}
\caption{Adaptive strategy vs.~convergence for fixed $\alpha_k$ from the sequence. The y-axis represents the error in $\log_{10}$ scale whereas the vertical lines represent the $m_k$ sequence starting with $m_0 = 10\Delta T$. The solid line is the convergence history with the sequence of $\alpha_k$ given as $(6.19 \times 10^{-7}, \, 5.56 \times 10^{-4}, \, 1.67 \times 10^{-2}, \, 9.13 \times 10^{-2} )$.}
\label{fig:advection64}
\end{figure}

Understanding the convergence behavior allows us to roughly predict the number of iterations for a given threshold. 
In order to examine the convergence of the method, we have to test it when solving the same initial value problem with the same domain for three thresholds $10^{-5}, \, 10^{-9}, \, 10^{-12}$.
One challenge for an actual application is the question of when to stop the iterations since the actual error is not available.
As discussed in section~\ref{sec:bounds}, a valid candidate is comparing successive iterates at the last time-step, cf. lemma~\ref{lemma:ultimate}.
To check the impact of this choice, figure~\ref{fig:2eqconv} shows the convergence behavior for the two test problems with the adaptive-$\alpha$ strategy, both for the actual errors and for the difference between successive iterates.
The discretization parameters for each threshold are chosen accordingly for a fixed $T$ (see table~\ref{tab:table1} for details).

\begin{table}[!htb]\footnotesize
\begin{center}
\begin{tabular}{l|lll|lll}
 & \multicolumn{3}{|l|}{Heat, $T = 0.1$} & \multicolumn{3}{|l}{Advection, $T = 10^{-2}$} \\
\hline
 \rule{0pt}{2.5ex}tol. to reach &$10^{-5}$&$10^{-9}$&$10^{-12}$&$10^{-5}$&$10^{-9}$&$10^{-12}$ \\
 no. of spatial points $N$ & 450 & 400 & 300 &  350 & 350 & 600  \\
 order in space $\kappa$ & 2 & 4 & 6 & 2 & 4 & 5 \\
 no. of collocation points $M$ & 1 & 2 & 3 & 2 & 2 & 3  \\
 no. of time steps $L$ & 32 & 32 & 16 & 8 & 16 & 32 \\
 linear solver tolerance $\tau$& $10^{-6}$ & $10^{-10}$ & $10^{-13}$ & $10^{-8}$ & $10^{-11}$ & $10^{-14}$
\end{tabular}
\end{center}
\caption{Parameter choice for solving the heat and advection equation in order to reach an error $\|\vect{u}(T) - \vect{u}_L\|_\infty < \operatorname{tol}$ when solving with a standard sequential approach. $\kappa$ denotes the discretization order in space, where upwind was chosen for the advection equation and centered differences for the heat equation.}
\label{tab:table1}
\end{table}

\begin{figure}[!htb]
\centering
\includegraphics[width=0.99\textwidth]{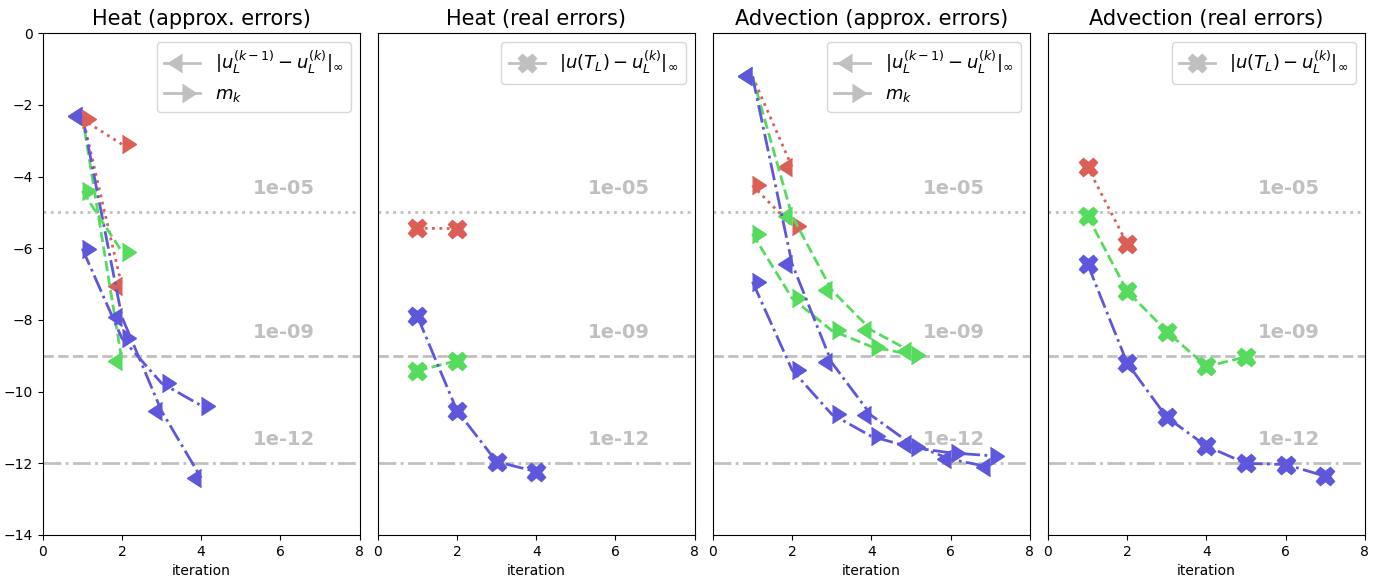}
\caption{Convergence with the adaptive strategy for different thresholds. The y-axis represents the error in $\log_{10}$ scale and vertical lines represent the thresholds for the stopping criteria: red for reaching thresholds under $10^{-5}$, green for $10^{-9}$ and blue for $10^{-12}$. 'approx. errors' graphs contain the information available on runtime which is: the errors of consecutive iterates in the last time-step (marker pointing left) and the approximations of the upper bound for the error in each iteration (marker pointing right), the $m_k$ values starting with $m_0 = \Delta T$.
These values are generated in Algorithm~\ref{alg:alphas} alongside the corresponding $\alpha_k$.
'real errors' graphs show corresponding errors to the exact solution, which is in general unavailable at runtime.
These errors are here as a proof of concept that the $m_k$ values are indeed following the real errors very well.
However, values $m_k$ should be treated in combination with the consecutive iterates, since they are an overestimate.}\label{fig:2eqconv}
\end{figure}

\subsection{Parallel scaling}
We now measure the actual speedup of the method for our two equations and the three different thresholds $10^{-5}$, $10^{-9}$, $10^{-12}$.
The idea is to compare the execution time of $64$ sequential steps solving the collocation problem via diagonalization of $\vect{Q}$ to the parallel method, handling $L =1,2,..., 64$ time-steps simultaneously. 
More precisely, our method is used as a parallel solver for a block of $L$ steps and repeatedly applied sequentially, like a moving window, until capturing all the $64$ time steps.
For each run, $\Delta T = T_{64}/64$ was fixed while $L$ varies (see table~\ref{tab:table2} for details).
The number of spatial and temporal discretization points is chosen in a way so that the method does not over-resolve in space or time (see table~\ref{tab:table2} for details).
Thus, the tolerance is coupled to $M$, simply because there is no point in trying to achieve low errors without deploying a higher-order method both in space and time~\cite{MinionEtAl2015}.
Note that the timings do not include the startup time, the setup, nor the output times.

\begin{table}[!htb]\footnotesize
\begin{center}
\begin{tabular}{l|lll|lll}
 & \multicolumn{3}{|l|}{Heat} & \multicolumn{3}{|l}{Advection} \\
\hline
 \rule{0pt}{2.5ex}tol. to reach &$10^{-5}$&$10^{-9}$&$10^{-12}$&$10^{-5}$&$10^{-9}$&$10^{-12}$ \\
 no. of spatial points $N$ & 350 & 400 & 350 & 800 & 800 & 700 \\
  order in space $\kappa$ & 2 & 4 & 6 & 1 & 3 & 5  \\
 no. of collocation points $M$ & 1 & 2 & 3 & 1 & 2 & 3  \\
 time endpoint $T_{64}$ & $0.32$ & $0.16$ & $0.16$ & $0.00016$ & $0.00064$ & $0.0128$ \\
 linear solver tolerance $\tau$& $10^{-6}$ & $10^{-10}$ & $10^{-14}$ & $10^{-6}$ & $10^{-11}$ & $10^{-14}$ \\
 linear solver tolerance $\widetilde{\tau}$& $10^{-6}$ & $10^{-10}$ & $10^{-13}$ & $10^{-9}$ & $10^{-13}$ & $10^{-15}$ \\
\end{tabular}
\end{center}
\caption{Parameter choice depending on the equation in order to reach an error $\|\vect{u}(T) - \vect{u}_L\|_\infty < \operatorname{tol}$ when solving with a standard sequential approach. Here $\kappa$ denotes the discretization order in space, where upwind was chosen for the advection equation and centered differences for the heat equation. Here $T_{64}$ represents the interval length that is needed so that the error is below $\operatorname{tol}$ after $64$ time-steps.}
\label{tab:table2}
\end{table}

The linear system solver used for the test runs is GMRES with a relative stopping tolerance $\tau$. 
An advantage of using an iterative solver in a sequential run is that $\tau$ can be just a bit smaller than the desired threshold we want to reach on our domain, otherwise, it would unnecessarily prolong the runtime. 
However, this is not the case when choosing a relative tolerance $\widetilde{\tau}$ for the linear solver within our method for inner systems (see algorithm~\ref{alg:alphas}, figure~\ref{fig:advection64} and section~\ref{sec:sande}). 
On one hand, the method strongly benefits from having a small $\widetilde{\tau}$ in general since it plays a key role in the convergence, as mentioned in lemma~\ref{lemma:ultimate}, but on the other hand, the linear solver needs more iterations, thus execution time $T_\mathrm{sol, par}$ is higher, cf.~\ref{sec:sande}. 
As a result, this is indeed a drawback when using the parallel method and has to be kept in mind.

\begin{figure}[!htb]
\centering
\subfigure[Heat equation]{%
	\includegraphics[height=1.8in]{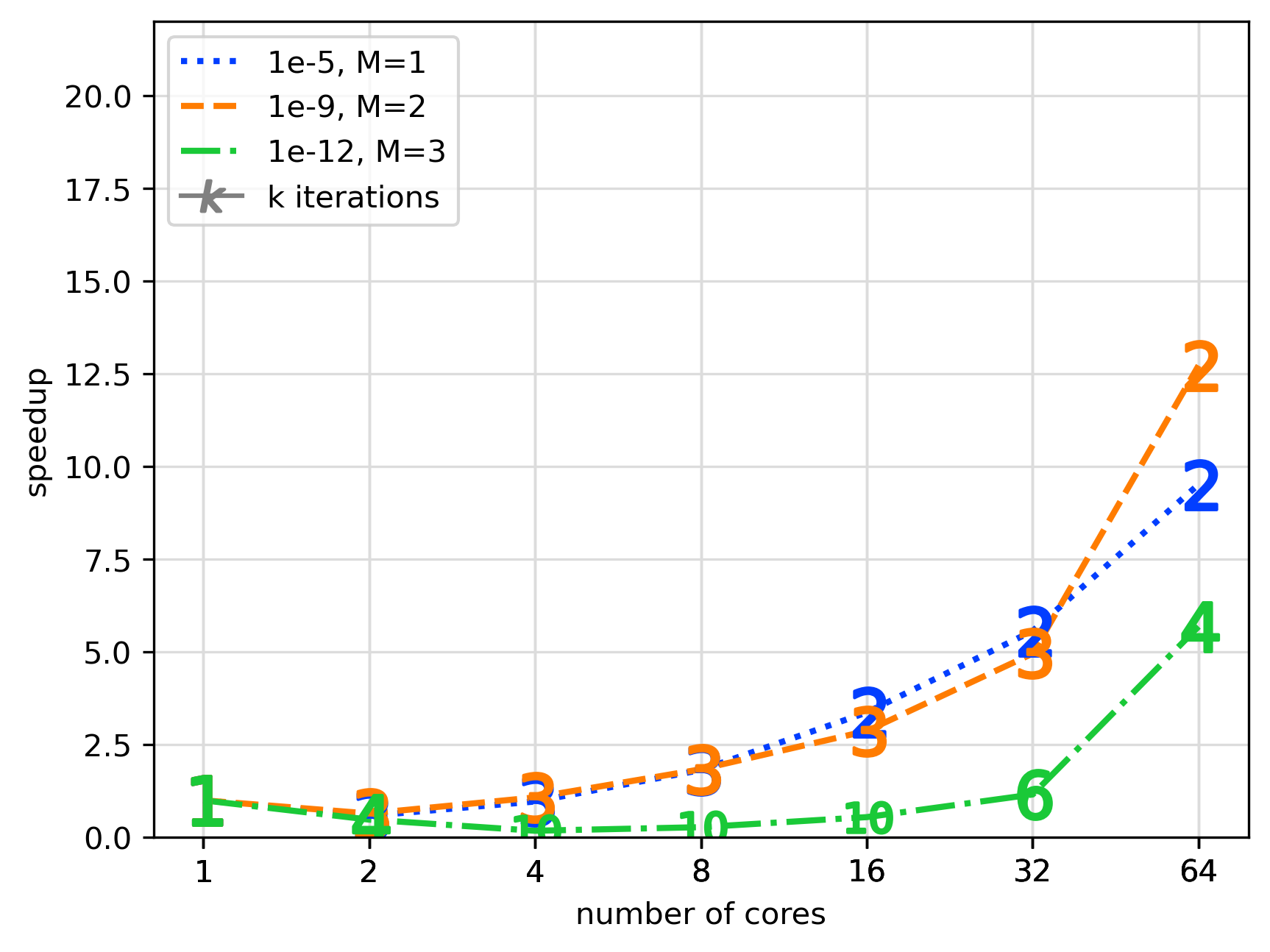}%
	\includegraphics[height=1.8in]{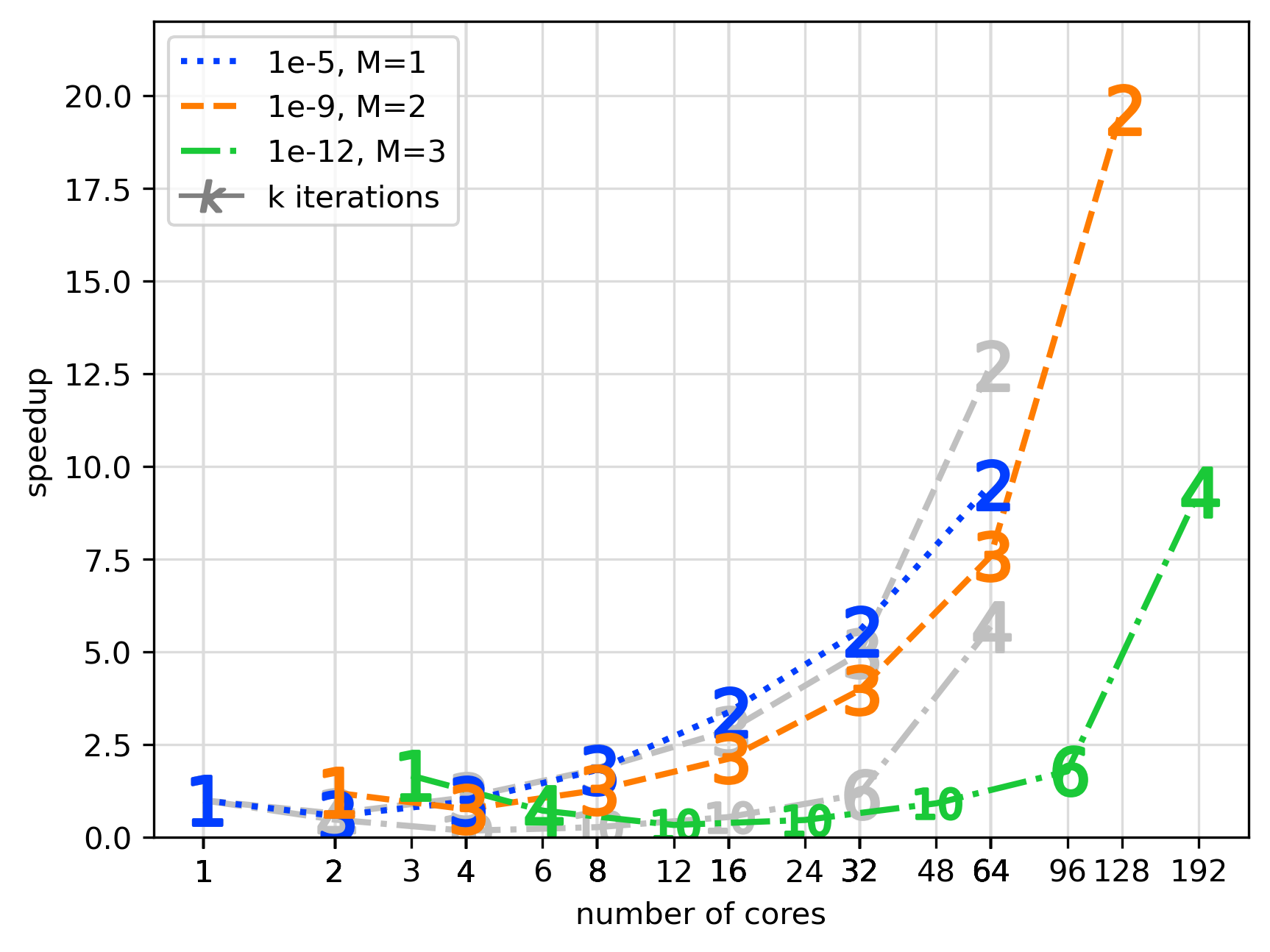}}%

\centering
\subfigure[Advection equation]{%
	\includegraphics[height=1.8in]{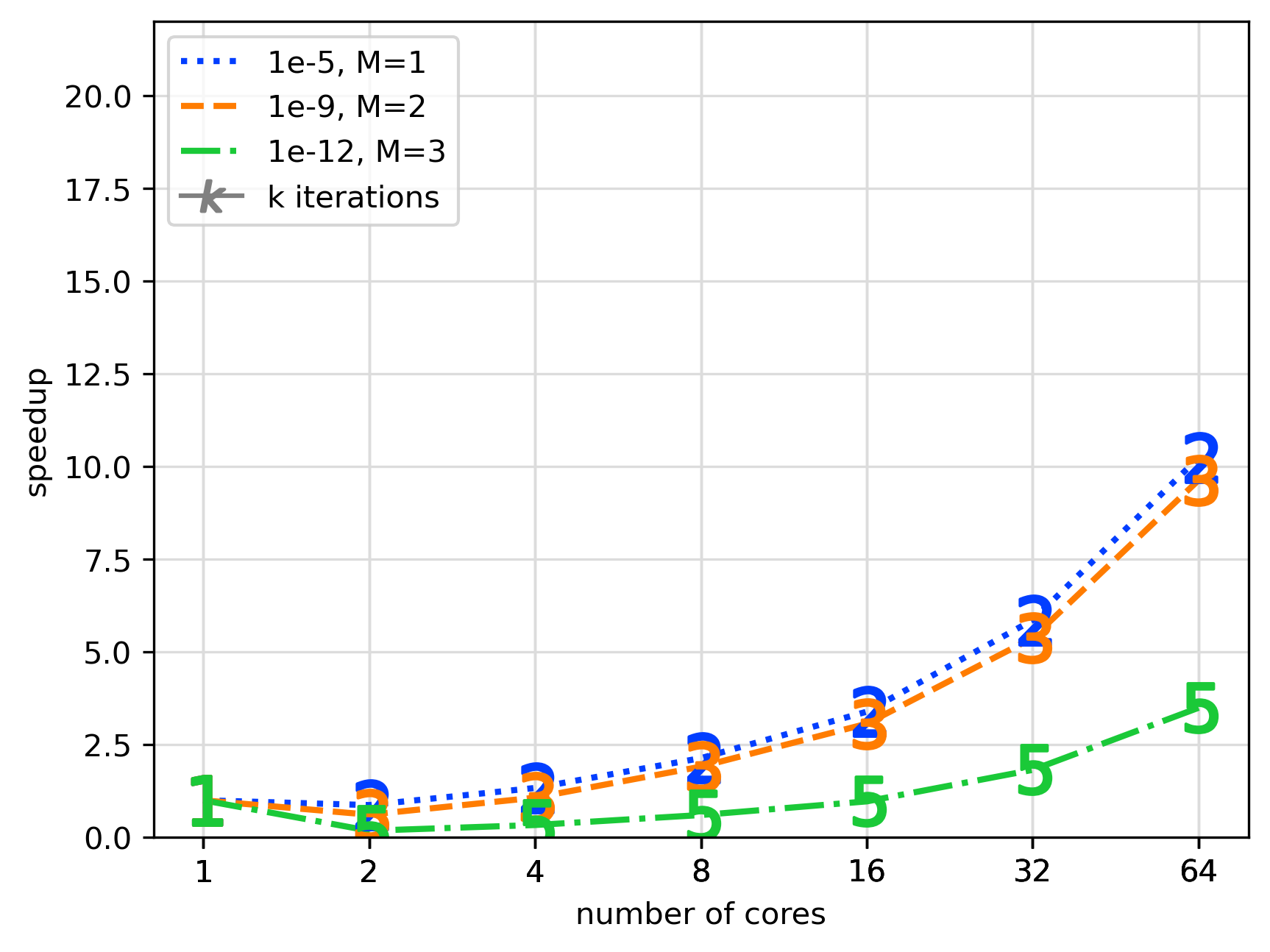}%
	\includegraphics[height=1.8in]{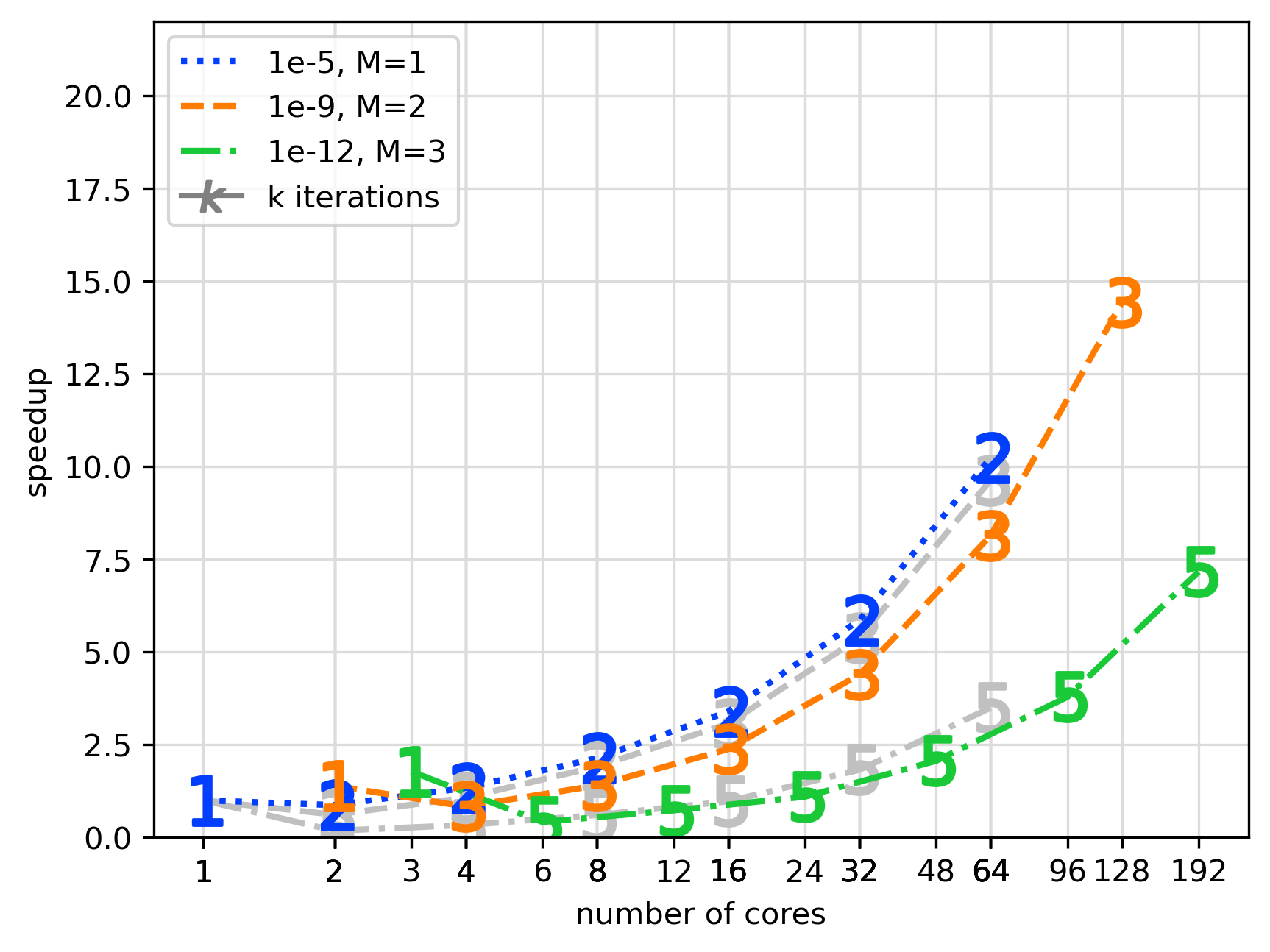}}%

\caption{Strong scaling plots for the three thresholds and two equations. The numbers on the curves represent the number of outer iterations Algorithm~\ref{alg:paralpha} needs in order to reach the given tolerance. 
The left graphs show a setup where $L = \nstep$ and $\ncoll=1$. 
This is a setting where the diagonalization of the preconditioner is handled in parallel and the inner systems are solved on only one core.
On the other hand, the graphs on the right are for $L = \nstep$ and $M=\ncoll$, showing that additional parallelism across the method improves speedup.
In both cases $\nspace = 1$, i.e.~there is no spatial parallelization. 
The gray plots are the same as in the left column, serving as reference information.\label{fig:strongscaling}}
\end{figure}

The strong scaling plots with parallelism across time-steps and across the collocation nodes are presented in figure~\ref{fig:strongscaling}. 
We can see that the method provides significant speedup for both heat and advection equations.
It also becomes clear that when very accurate results are needed (green curves), the performance is degrading.
Interestingly, this is in contrast to methods like PFASST or RIDC~\cite{RIDC}, where higher order in time gives a better parallel performance.
Comparing parallelization strategies, using both parallelizations across the collocation nodes and time-steps leads to better results and is always preferable.
The results do not differ much between the heat and the advection equation, showing only slightly worse results for the latter.
This supports the idea that single-level diagonalization as done here is a promising strategy also for hyperbolic problems.
Note, however, that for the heat equation we have multiple runs with up to $10$ iterations.
This is because the linear system shifts that are produced by the diagonalization procedure are poorly conditioned for the particular choice of the $\alpha_k$. 
Manually picking this sequence with slightly different, larger, values can circumvent this issue to a certain degree.
However, we did not do this here, because we wanted to show that the method does not always perform well when used without manually tweaking the parameters.
A more in-depth study of this phenomenon is needed and left for future work.
Note that this is not an issue for direct solvers.

Yet, parallel-in-time integration methods are ideally used in combination with a space-parallel algorithm, especially in the field of PDE solvers. 
Therefore, we test the method together with \texttt{petsc4py}'s parallel implementation of GMRES for the advection equation.
The results are shown in figure~\ref{fig:strongscaling2}.
We scaled \texttt{petsc4py} using up to $96$ cores.
This was done by solving sequentially in time with implicit Euler on $64$ time-steps.
The number of cores $\nspace=12$ is chosen as the last point where \texttt{petsc4py} scaled reasonably well for this problem size.
After fixing the spatial parallelization, the double time parallelization is layered on top of that.
The strong scaling across all quadrature nodes and time-steps is repeated for $L =1, 2,..., 64$ for three different thresholds $10^{-5}, \, 10^{-9}, \, 10^{-12}$. 
The plots show clearly that by using this method we can get significantly higher speedups for a fixed-size problem than when using a space-parallel solver only.
In the best case presented here, we obtain a speedup of about $85$ over the sequential run.
We would like to emphasize here that all runs are done with realistic parameters, not over-resolving in space, nor time, nor in the inner solves~\cite{12ways}.
Thus, while the space-parallel solver gives a speedup of up to about $8$, we can get a multiplicative factor of more than $10$ by using a space- and doubly time-parallel method.

\begin{figure}[!htb]
\centering
\includegraphics[width=0.8\textwidth]{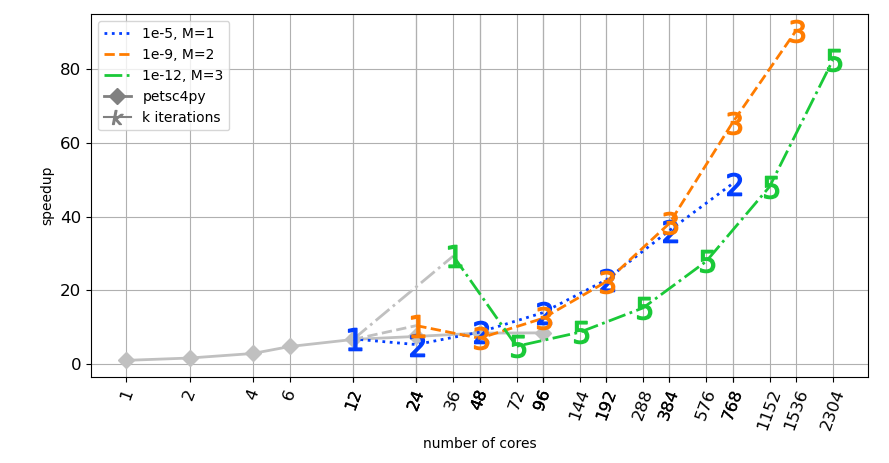}
\caption{Strong scaling plots for the advection equation and three thresholds.
The solid gray line represents the spatial scaling with \texttt{petsc4py} for the advection equation solved sequentially with implicit Euler on $64$ time-steps. 
The curve shows the scaling of \texttt{petsc4py} for our problem is best around 12 cores, since using more cores does not increase the speedup much. 
The colored lines represent the scaling where $\nstep=L$, $\ncoll=M$, and $\nspace=12$, in other words, parallelism across time-steps, across the method, and in space. 
The numbers on the curves represent the number of outer iterations Algorithm~\ref{alg:paralpha}} needs in order to reach the given tolerance. \label{fig:strongscaling2}
\end{figure}

%% file: section6.tex
\section{Conclusion}
In this paper, an analysis and an implementation of a diagonalization-based parallel-in-time integrator for linear problems is presented.
Using an $\alpha$-circulant, "all-at-once" preconditioner within a simple Richardson iteration, the diagonalization of this preconditioner leads to an appealing time-parallel method without the need to find a suitable coarsening strategy.
Convergence is very fast in many cases so that this does not even require an outer Krylov solver to get an efficient parallel-in-time method.
We extend this idea to high-order collocation problems and show a way to solve the local problems on each time-step efficiently again in parallel, making this method doubly parallel-in-time.
Based on the convergence theory and error bounds, we propose an effective and applicable strategy to adaptively select the crucial $\alpha$-parameter for each iteration.
However, depending on the size of the collocation problem, we show that some of these $\alpha$ lead to non-diagonalizable inner systems, which need to be avoided.
This theoretical part is augmented by a thorough description of the parallel implementation.
We estimate the expected speedup, provide verification of the adaptive choice of the parameters and, finally, show actual parallel runs on a high-performance computing system on up to $2304$ cores.
These parallel tests demonstrate that our proposed approach does indeed yield a significant decrease in time-to-solution, even far beyond the point of saturation of spatial parallelization.

By design, the diagonalization-based parallel-in-time integrators work only for linear problems with constant coefficients.
In order to solve more complex, more realistic problems, they have to be coupled to a nonlinear solver.
Here, they work as inner solvers for an inexact outer Newton iteration.
As such, the challenges and solutions presented here are still valid. 
In future work, we will extend this implementation and the results to nonlinear problems, using the existing theory from the literature and our own extensions as described in this paper.
Although not required, coupling the adaptive strategy with outer iterations of a Krylov solver is an interesting field of future research. 

%% file: appendix.tex
\section{Proof of lemma~\ref{lemma:ultimate}}\label{app:proof}

With $\tC z = w$ it holds
\begin{align*}
\Delta w = \tC \Delta z + \Delta \tC (z + \Delta z).
\end{align*}
\alert{Using the triangle inequality, we have}
\begin{align}\label{eq:dw}
\|\Delta w\| \leq \|\tC\| \|\Delta z\| + \varepsilon\|\tC\| \|\Delta z + z\|.
\end{align}
Now we need to find a way to bound $\|\Delta z\|$ and $\|\Delta z + z\|$. 
Since the system solving is inexact, there exists a vector $\xi$ such that $(\tB + \Delta \tB) (z + \Delta z) = y + \Delta y + \xi$ \alert{with} $\|\xi\| \leq \tau \|y + \Delta y\|.$ \alert{Since} $(\tB + \Delta \tB)\Delta z + \Delta \tB z = \Delta y + \xi$,
\begin{align*}
\Delta z = (\tB + \Delta \tB)^{-1}(\Delta y + \xi - \Delta \tB z) 
\end{align*}
which yields the bound 
\begin{align}\label{eq:zz}
\|\Delta z\| &\leq \|(\tB + \Delta \tB)^{-1}\|\big(\|\Delta y\| + \|\xi\| + \varepsilon\|\tB\| \|z\|\big) \notag\\
		& \leq \|(\tB + \Delta \tB)^{-1}\|\big(\|\Delta y\| + \tau\|\Delta y + y\| + \varepsilon\kappa(\tB) \|y\|\big),
\end{align}
where the last inequality comes from the fact that $\|z\| \leq \|\tB^{-1}\|\|y\|$. 
Note that $\tB + \Delta \tB$ is invertible for small perturbations if $\tB$ is invertible. On the other hand, we have
\begin{align*}
z + \Delta z = (\tB + \Delta \tB)^{-1}(y + \Delta y + \xi)
\end{align*}
which gives
\begin{align} \label{eq:zdz}
\|z + \Delta z\| &\leq \|(\tB + \Delta \tB)^{-1}\| \big( \|y + \Delta y\| + \|\xi\|\big) \notag\\
		 & \leq (1 + \tau)\|(\tB + \Delta \tB)^{-1}\| \|y + \Delta y\|.
\end{align}
Combining~\eqref{eq:zz} and~\eqref{eq:zdz} with~\eqref{eq:dw} yields
\begin{align}\label{eq:dw2}
\|\Delta w\| \leq \|(\tB + \Delta \tB)^{-1}\| \|\tC\|\big( \|\Delta y\| + (\tau + \varepsilon + \varepsilon \tau)\|y + \Delta y\| + \varepsilon \kappa(\tB)\|y\| \big).
\end{align}
Since $y = \tA x$ and $\|\Delta \tA\| \leq \varepsilon \|\tA\|$, we get the inequalities
\begin{align*}
\|y\| &\leq \|\tA\|\|x\|, \\
\|\Delta y\| &= \|\Delta \tA x\| \leq \varepsilon \|\tA\|\|x\|, \\
\|y + \Delta y\| &\leq (1+\varepsilon)\|\tA\|\|x\|.
\end{align*}
Including these inequalities in~\eqref{eq:dw2} yields
\begin{align*} 
\|\Delta w\| &\leq \|(\tB + \Delta \tB)^{-1}\| \|\tC\|\|\tA\|\big( \varepsilon + (\tau + \varepsilon + \varepsilon \tau)(1+\varepsilon) + \varepsilon \kappa(\tB)\big)\|x\| \\
& \leq \|(\tB + \Delta \tB)^{-1}\| \|\tC\|\|\tA\|\big( 2\varepsilon + \tau + \varepsilon \kappa(\tB)\big)\|x\| + O(\varepsilon \tau + \varepsilon^2).
\end{align*}
It remains to bound $\|(\tB + \Delta \tB)^{-1}\|$. 
From $(\tB + \Delta \tB)^{-1} = \tB^{-1}(\vect{I} + \tB^{-1} \Delta \tB)^{-1}$ we have
\begin{align*}
    \|(\tB + \Delta \tB)^{-1}\| \leq \|\tB^{-1}\|\|(\vect{I} + \tB^{-1} \Delta \tB)^{-1}\| \leq \frac{\|\tB^{-1}\|}{1 - \|\tB^{-1} \Delta \tB\|}.
\end{align*}
Since the function $1/(1-x)$ is monotonically increasing, we get
\begin{align*}
    \|(\tB + \Delta \tB)^{-1}\| \leq \frac{\|\tB^{-1}\|}{1 - \varepsilon \kappa(\tB)}
\end{align*}
which completes the proof.